\documentclass[12pt]{article}
\usepackage{graphicx}
\usepackage{amssymb,amsmath,amsthm}
\newtheorem{definition}{Definition}[section]
\newtheorem{theorem}[definition]{Theorem}
\newtheorem{lemma}[definition]{Lemma}

\newtheorem{proposition}[definition]{Proposition}
\newtheorem{remark}[definition]{Remark}

\numberwithin{equation}{section}
\title{ The Geometric Approach to the Classification of Signals via a Maximal Set of Signals.}

\author{Leon A. Luxemburg, \thanks{Department of Mathematics, Texas A\& M University at Galveston, Texas, U.S.A.: luxembul@tamug.edu}, 
and Steven B. Damelin \thanks{Department of Mathematics, University of Michigan, 530 Church Street, Ann Arbor, MI 48109, USA: damelin@umich.edu}}

\begin{document}
\maketitle


\begin{abstract}
In this paper we study the scale-space classification of signals via the maximal set of kernels. We use a geometric approach which arises naturally when we consider parameter variations in scale-space. We derive the Fourier transform formulas for quick and efficient computation of zero-crossings and the corresponding classifying trees. General theory of convergence for convolutions is developed, and practically useful properties of scale-space classification are derived as a consequence.We also give a complete topological description of level curves for convolutions of signals with the maximal set of kernels. We use these results to develop a bifurcation theory for the curves under the parameter changes. This approach leads to a novel set of integer invariants for arbitrary signals..
\end{abstract}


\section{Introduction}
  \label{sec1}
Scale-space filtering is a very useful method for classification, recognition and structural feature extraction of waveforms. It is based on a convolution $\phi(x,\rho)$ of a given signal $f(x)$ with a kernel $g(x,\rho)$
\begin{equation}
  \label{eq:1.1}
  \phi(x,\rho) = f(x)*g(x,\rho) \equiv \int_{-\infty}^\infty f(x-\nu)g(\nu,\rho)\,d\nu\, .
\end{equation}
After the function $\phi(x,\rho)$ is defined one can construct a Witkin tree $TW(f)$ corresponding to the zero-crossing curves $\partial\phi(x,\rho)=0$. See [1,2]. In [3] we also defined a procedure for constructing a topologically invariant tree   $TT(f)$. In [3] we constructed a maximal family $M$ of kernels $g(x,\rho)$ such that the construction of Witkin trees leads to a meaningful shape invariant classification of signals. Before this result, the only kernels used for scale-space filtering were Gaussian kernels $g(x,\rho)={y\over\sqrt{2\,\!\pi}}e^{-(x\rho)^2/2}$. The set of kernels $M$ consists of kernels
\begin{equation}
  \label{eq:1.2}
%
  \Lambda_{\alpha\beta p}(x,\rho)=\alpha \rho^{p+1}\Lambda_p^o(x\rho)+\beta \rho^{p+1}\Lambda_p^e(x\rho)\, ,
\end{equation}
 where $\alpha$, $\beta$ and $p$ are real numbers, $\Lambda_p^o(z)$ and $\Lambda_p^e(z)$ are odd and even functions defined by
 \begin{eqnarray}
  \Lambda_p^e(z) &=& {1\over\sqrt{2\,\!\pi}}\sum_{n=0}^\infty\left[ {z^{2n}\over (2n)!}(-1)^n\prod_{k=0}^{n-1}(p+2k+1)\right] \label{eq:1.3} \\
   \Lambda_p^o(z) &=& {1\over\sqrt{2\,\!\pi}}\sum_{n=0}^\infty\left[ {z^{2n+1}\over (2n+1)!}(-1)^n\prod_{k=0}^{n}(p+2k)\right]\, . \label{eq:1.4}
 \end{eqnarray}
For a complete treatment of the subject of maximal set of kernels as well as of the motivations behind the scale-space approach, see our previous paper [3]. 

In the present paper we consider and solve the following important problems related to the study of the maximal set of kernels:
\medskip

\noindent {\bf 1)}
We derive formulas for the Fourier transforms of the kernels $\Lambda_{\alpha\beta p}(x,\rho)$. This gives us a quick and efficient tool for constructing zero-crossings of the corresponding function $\phi(x,\rho)$, and therefore for the construction of the tree.
\medskip

\noindent {\bf 2)}
We prove a number of properties concerning the convergence of functions $\phi(x,\rho)$ and their partial derivatives, where $\phi(x,\rho)$ is defined as a convolution of a signal $f(x)$ with a kernel $\Lambda_{\alpha\beta p}(x,\rho)$. We develop a general theory of convergence of  kernels.
\medskip

\noindent {\bf 3)}
Using results in (1) and (2) above, we give a complete topological description of level curves $\phi_x(x,\rho)=c$ as well as level curves of arbitrary partial derivatives of $\phi(x,\rho)$.
\medskip

\noindent {\bf 4)}
If we consider a convolution $\Lambda_{\alpha\beta p}(x,\rho)=f(x)*\Lambda_{\alpha\beta p}(x,\rho)$ of a signal $f(x)$ with a kernel $\Lambda_{\alpha\beta p}(x,\rho)$ then as $\alpha$, $\beta$ and $p$ vary, this results in different zero crossings and different trees. We study bifurcations of the curves under parameter change. This gives us useful integer invariants for the signal which provides another way to classify the signals. We also use methods of differential and algebraic topology to derive some global formulas for tree invariants.  

\section{Properties of Convolution Kernels}
  \label{sec2}
Let $g_p^\theta = \rho^{p+1}\Lambda_p^\theta(x\rho)$ where $\theta = e,o$. Then, in order to be able to compute the convolution $\phi_p^\theta(x, \rho)=f(x)*g_p^\theta(x,\rho)$ more efficiently, we need to know its Fourier transform. This will allow us to perform computations in the frequency domain. Another question closely related to this is a question of finding the limiting function $\phi_p^\theta(x,\rho)$ as $\rho\to \infty$. It is well known that for the Gaussian $g_0^e(x,\rho)$, the Fourier transform is $c\,\!e^{s^2/\rho^2}$ where $c$ is some constant. As $\rho\to \infty$, $c\,\!e^{s^2/\rho^2}\to c$ ; thus, when we perform the convolution, we get the same signal $f(x)$ in the limit (this is due to the fact that the $\delta$ - function has the Fourier transform equal to 1). Therefore, the limiting value for $\phi_0^\theta(x,\rho)$ as $\rho\to\infty$ is the function $f(x)$ itself, i.e., as $\rho$ increases, the complexity of the convolution  $\phi_0^\theta(x,\rho)$ increases until we come to a signal $f(x) = \phi_0(x,\infty)$ in the limit. It turns out that if $n$ is an integer, the limiting function for $\phi_n^\theta(x,\rho)$ as $\rho\to\infty $ is the $n$th order derivative of $f(x)$, where $\pi$ is odd or even, depending on whether $n$ is odd or even. 

We will generalize this simple property for every $p$ by finding the limit $\phi_p^\theta(x,\rho)$ as $\rho\to\infty$ , proving its existence and showing that it is in some sense a fractional $p$th order derivative of the signal. We will also consider here a number of very useful properties of convolutions $g_p^\theta(x,\rho)$.
\begin{remark}\label{rem2.1}
Since the trees of the signals do not change if we scale it or multiply it by a constant, we will determine Fourier transforms only up to a constant factor and will not spend any time on determining these factors precisely.  
\end{remark}
In what follows, a function $f(x)$ and its Fourier transform $\mathcal{F}(f)$ will be called a \textit{Fourier pair} if they are connected by the usual direct and inverse Fourier transform formulas. In this case, we will write $f~\mathcal{F}(f)$. In general, the direct Fourier transform formula does not imply the inverse, i.e., if $\mathcal{F}(f)=F(s)$, it is not generally true that $f(x) =\mathcal{F}^{-1}\big(F(s)\big)$.  
\begin{theorem}\label{thm2.1}
If $p>0$, then the Fourier transforms $F_p^\theta(\omega)$ of $\Lambda_p^\theta(x)$, $\theta=e,o$ exist, and
\begin{equation}
  \label{eq:2.1}
  F_p^e(\omega)=c\,\!|\omega|^p\,\!e^{-\omega^2/2}\, ,\qquad F_p^o(\omega)=d\,\!j\, sgn(\omega)|\omega|^p\,\!e^{-\omega^2/2}\, .
\end{equation}
 where $c$, $x$ are some real numbers not equal to zero. Also, $F_0^e(\omega)=c_1\,\!|\omega|^p\,\!e^{-\omega^2/2}$.
  \end{theorem}
  \begin{proof}
The existence of the Fourier transform is guaranteed by the absolute integrability of functions $\Lambda_p^\theta(x)$ for $p>0$ (see [3]) and by their differentiability (see [4]), and Theorem 2.1.

Let $\Lambda_p^\theta(x)=h(x)$. Then, for every $x$ it satisfies the equation (see [3] ) 
\begin{equation}
  \label{eq:2.2}
  h^{\prime\prime}(x)+x\,\!h^\prime(x)+(p\!+\!1)h(x)=0\, .
\end{equation}
Formulas (3.5) and (3.6) in [3] guarantee the existence of Fourier transforms of each summand in (\ref{eq:2.2}), and its convergence to zero as $|x|\to\infty$. Applying a Fourier transform to (\ref{eq:2.2}), gives us
\begin{equation}
  \label{eq:2.3}
  s^2H(s)-{d\over ds}\left(sH(s)\right)+(p\!+\!1)H(s)=0
\end{equation}
where $H(s)= \mathcal{F}\big(h(t)\big)\,$,  $s=j\,\!\omega$, and $\omega$ is real.
 
 Here we used the fact that if $\mathcal{F}(g) = G(s)$ , then $\mathcal{F}\big( g(t)\cdot t\big) = - {dG(s)\over ds}$ provided $\int_{\infty}^\infty t\,\!g(t)e^{-st} dt$ converges uniformly in $s$ and  $\int_{\infty}^\infty g(t)e^{-st} dt$ converges (see [5], 11.55 a). In our case, $g(t)  = h^\prime(t)$ and the conditions on convergence follow from (3.5) in [3], and from the fact that $p>0$. From (\ref{eq:2.3}), it follows that  ${H^\prime(s)\over H(s)}= s + {p\over s}$. Integrating this along $j\omega$ we get
 \begin{equation}
   \label{eq:2.4}
   H(s) = c_1 s^p e^{s^2/2}\, .
 \end{equation}
The path of integration should not include $0$, since we have a singular point there. Thus, if $s_0=j\omega_0$, $\omega_0>0$, formula (\ref{eq:2.4}) holds for $s=j\omega$ with $\omega>0$ only. Since $H(s)$ must be real because $h(t)$ is an even function, $H(s) = c\omega^pe^{-\omega^2/2}$ for some real $\omega>0$ and $s = j\omega$. Now, since $H(-s)=H(s)$ for an even $h(x)$, we have $H(s)=c|\omega|^pe^{-\omega^2/2}$ for $s=j\omega$, $w<0$. Since $F_p^e(\omega)=H(s)$, this proves the first part of (\ref{eq:2.1}). The second formula in (\ref{eq:2.1}) can be proved in an entirely similar fashion. The last equation follows from the fact that $\Lambda_0^e(x)$ is Gaussian.
\end{proof}
 \begin{remark}\label{rem2.2} 
Formulas (\ref{eq:2.1}) are extremely useful for fast frequency domain computerized computation of convolution $f(x)*g_p^\theta(x,\rho)$ and therefore for constructing trees. Formulas (\ref{eq:2.1}) can be rewritten in an equivalent form:
For $p\neq 2n+1$, $\theta=e$ and for $p\neq 2n$, $\theta=o$, we have:
\begin{equation}
  \label{eq:2.5}
F_p^e(s)=c\big(s^p+(-s)^p\big)e^{-s^2/2}\,,\qquad F_p^o(s)=d\big(s^p-(-s)^p\big)e^{-s^2/2}\, \, \,s=j\omega\, ,
\end{equation}
where $c$, $d$ are real numbers not equal to zero. 
\end{remark}
\medskip

Before we go further let us recall a definition of fractional derivatives. Let $q<0$, then [6] 
\begin{equation}
  \label{eq:2.6}
  {d^q f(x)\over dx^q} = {1\over \Gamma(-q)}\int_0^x{f(\rho)\,dy\over(x-\rho)^{q+1}} 
                      = {1\over \Gamma(-q)}\int_0^x{f(x-\rho)\over \rho^{q+1} }\,dy\, .
\end{equation}
This definition of fractional derivatives belongs to Liouville. We propose a related but slightly different definition. Let $\mathcal{F}(f)=F(s)$; then for $s=j\omega$,
\begin{equation}
  \label{eq:2.7}
  {d^p f(x)\over dx^p}=f^{(p)}(x)={c\over 2\pi j}\int_{-\infty}^\infty s^pF(s)e^{sx}\,ds\, .
\end{equation}
We put a constant real number in (\ref{eq:2.7}) as usual to emphasize that we are only interested in functions up to a scalar factor. In what follows, we will use the following simple fact from the theory of Fourier transforms [4]:
\begin{proposition}\label{prop:2.1}
If a function $f(x)$ is continuously differentiable with $i$th order derivatives absolutely integrable for all $i$, $0\!\leq \!i\! \leq \!m $, then for some $c>0$,
\begin{equation}
  \label{eq:2.8}
  \left| \mathcal{F}\big(f(x)\big)(s)\right| = |F(s)|\leq {c\over|s^m|}\, .
\end{equation}
\end{proposition}
Let $\mu_\alpha(t) = t^{-\alpha}$ for $t>0$ and  $\mu(t)=0$ for $t\leq 0$ where $0<\alpha<1$; then $\mathcal{F}\big(\mu(t)\big)(s)=M(s)=s^{\alpha-1}\Gamma(1\!-\!\alpha)$. Therefore, by the convolution theorem we have:
\begin{align}
f^{(p)}(x)=\mathcal{F}^{-1}\big(c\,\!s^pF(s)\big) &=   \notag \\ 
\mathcal{F}^{-1}\big(c\,\!s^p\big)*f(x) &= c\,\!\mu_{p+1}(x)*f(x)= c\,\!\int_0^\infty{f(x-\rho)\over \rho^{p+1} }\,dy \, .\notag
\end{align}
If, in addition, $f(\rho)=0$ for $\rho<0$ then $f^{(p)}(x) = c\,\!\int_0^x{f(x-\rho)\over \rho^{p+1}}\,dy$ and this coincides with classical definition (\ref{eq:2.6}). However, the formulas we used above require certain convergence properties which are not satisfied for every $f(x)$ and $p$ . We also need to establish the range of applicability of definition (\ref{eq:2.7}) and show that it extends the usual definition of derivative when $p$ is an integer. Our purpose here is not to develop a fractional calculus, but to use some of its concepts to explain our results and their applications. As usual, we introduce the equivalence relation equating a function $f(x)$ with $c\,\!f(x)$ where $c$ is a constant not equal to zero. 
\begin{theorem}\label{thm2.2}
  If a function $f(x)$ satisfied \ref{prop:2.1} for $m>p+1$, then\newline
\medskip

 (i) Formula (\ref{eq:2.7}) gives us a correctly defined $f^\prime(x)$. \newline
\medskip

(ii) Also, if $x\,\!f(x)$ is absolutely integrable, then
\begin{equation}\label{eq:2.9}
  c\,\!s^pF(s)=\mathcal{F}\big(f^{(p)}(x)\big),\qquad  \mathcal{F}^{-1}\big(c\,\!s^pF(s)\big)=f^{(p)}(x)\, .
\end{equation}\newline
\medskip

(iii) Formula (\ref{eq:2.7}) gives us the usual derivative if $p$ is an integer. \newline
\medskip

(iv) If $p+q+1<m$, then ${d^p\over dx^p}\big({d^q\over dx^q}f(x)\big) = {d^{p+q}\over dx^{(p+q)}}f(x)$. 
\end{theorem}
\begin{proof}
(i): Follows immediately from (\ref{eq:2.8}) and the definition of $f^{(p)}(x)$ in (\ref{eq:2.7}).\newline
\medskip

(ii): If $x\,\!f(x)$ is absolutely integrable, then $F(s)$ is continuously differentiable and absolutely integrable (see [4]). Therefore, $s^p F(s)$ is continuously differentiable. It is  also absolutely integrable as is seen from (\ref{eq:2.8}) and the condition $m>p+1$. Since $2\pi i\,\!f^{(p)}(-x) = \mathcal{F}\big(c\,\!s^pF(s)\big)$ by (\ref{eq:2.7}), we can now apply the theorem on the inversion of the Fourier transform [4], and conclude that $\mathcal{F}\big(f^{(p)}(x)\big)=$ and $f^{(p)}(x)=\mathcal{F}^{-1}\big(c\,\!s^pF(s)\big)$\newline
\medskip

(iii): If $p$ is an integer, the assertion of (iii) is a very well known fact. \newline
\medskip

(iv): Follows immediately from (ii) and (\ref{eq:2.7}).
\end{proof}
\begin{remark} \label{rem2.4}
 Formulas (\ref{eq:2.5}) and Theorem~\ref{thm2.2} imply that for some $\alpha$ and $\beta$, linear combination $\alpha\Lambda_p^e(x)+\beta\Lambda_p^o(x)=\mu_p(x)$ is actually the $p$th fractional derivative of the Gaussian ${1\over\sqrt{2\,\!\pi}}e^{-(x)^2/2}  = G(x)$.
\end{remark}
This is a generalization of the observation made in Section 2 of [3] for the integer $p$ . Coefficients $\alpha$ and $\beta$ are easy to find. Thus, we can form convolutions $f(x)*\mu_p(x\rho)$ for every $\rho$ and $p$. As we mentioned before, the limit function for $f(x)*G(x\rho)$ as $\rho \to \infty$ is $f(x)$. It turns out that the limit function for $f(x)*\mu_p(x\rho)$ as $\rho\to \infty$ is $f^{(p)}(x)$. This is a very important generalization of the formula for $p=0$. This property follows from the following theorem. Before we state Theorem~\ref{thm2.3}, we need a preliminary lemma. 
\begin{lemma}\label{lem2.1}
Let $f(x)$ be an absolutely integrable and continuously differentiable function, and let\newline
\medskip  

(i) $p>0$ or $\theta=e$ and $p=0$, or \newline
\medskip

(ii) $p$ be arbitrary and $f(x)$ transient, i.e. zero outside a finite interval.  \newline
\medskip

Then,  \newline
\medskip

(a) $\phi_p^\theta(x,\rho)=f(x)*\rho^{p+1}\Lambda_p^e(x\rho)$ is correctly defined and absolutely integrable on $(-\infty,\infty)$ by $x$ and is infinitely differentiable by $x$ and $\rho$.  \newline
\medskip

(b)  The following functions constitute Fourier pairs.
\begin{align}
   \rho^{p+1}\Lambda_p^e(x\rho)\ &~ \ c\big( s^p+(-s)^p\big)e^{s^2/(2\rho^2)} \, \label{eq:2.10a}\, , \\
   \rho^{p+1}\Lambda_p^0(x\rho)\ &~ \ c\big( s^p-(-s)^p\big)e^{s^2/(2\rho^2)} \, \label{eq:2.10b}\, ,
\end{align}
\begin{align}
   \phi_p^e(x,\rho)\ &~ \ c\big( s^p+(-s)^p\big)e^{s^2/(2\rho^2)}F(s) \, \label{eq:2.11a}\, , \\
   \phi_p^0(x,\rho)\ &~ \ c\big( s^p-(-s)^p\big)e^{s^2/(2\rho^2)}F(s) \, \label{eq:2.11b}\, ,
\end{align} 
where $F(s)=\mathcal{F}\big(f(x)\big)$, $s=i\,\!\omega$.
\end{lemma}
\begin{proof}
(a) trivially follows from the theorem on differentiation of improper integrals by a parameter [5]. (b) follows from Theorem~\ref{thm2.1}.     
\end{proof}
\begin{theorem}\label{thm2.3}
If, for a function $f(x)$, conditions of Theorem~\ref{thm2.2} and either (i) or (ii) of  Lemma~\ref{lem2.1} are satisfied, then\newline
\medskip

(i) Functions $\phi _p^e(x,\rho)$ , $\theta=e,o$ converge uniformly  as $\rho\to\infty$ to a continuous function  $f_p^\theta(x)$ which is continuously differentiable $k$-times if $k<m-p-1$.  \newline
\medskip

(ii) For any $k$, $0<k<m-p-1$ functions ${d^k\,\!\phi _p^e(x,\rho) \over dx^k}$, $\theta=e,o$ converge uniformly to ${d^k \over dx^k}f_p^\theta(x)$ as $\rho\to\infty$.\newline
\medskip

(iii) For the functions $f_p^\theta(x)$, $\theta=e,o$ we have the following Fourier correspondence:  
\begin{align}
  f_p^e(x,\rho)\ &\sim\  \ c\big( s^p+(-s)^p\big)F(s) \, \notag\, , \\
  f_p^0(x,\rho)\ &\sim \ c\big( s^p-(-s)^p\big)F(s) \, \notag\, ,
\end{align} 
where $F(s)=\mathcal{F}\big(f(x)\big)]$.  
\end{theorem} 
\begin{proof}
From Lemma~\ref{lem2.1} it follows that for $s=i\,\!w$
\begin{equation}\label{eq:2.12}
  \phi_p^e(x,\rho)={c\over 2\pi i}\int_{-i\infty}^{i\infty} F(s)\big(s^p+(-s)^p\big)e^{s^2/(2\rho^2)}e^{sx}\,ds\, .
\end{equation}
Since $m>p+1$, inequality (\ref{eq:2.8}) implies absolute convergence of the integral  
\begin{equation}\label{eq:2.13}
  f_p^e(x)={c\over 2\pi i}\int_{-i\infty}^{i\infty} F(s)\big(s^p+(-s)^p\big)e^{sx}\,ds\, 
\end{equation}
for all $x\in(-\infty,\infty)$. Let us prove that uniformly on $(-\infty,\infty)$
\begin{equation}\label{eq:2.14}
  f_p^e(x)=\lim_{y\to\infty}\phi _p^e(x,\rho)\, .
\end{equation}
Let $\epsilon>0$. Since $m>p+1$, inequality (\ref{eq:2.8}) implies that $\exists T>0$ such that   
\begin{equation}\label{eq:2.15}
  \Big|{c\over 2\pi i}\Big|\Big(\int_{i\,\!T}^{i\infty}+\int_{-i\infty}^{-i\,\!T}\Big)  \Big|F(s)\big(s^p+(-s)^p\big)\,ds\Big|<{\epsilon \over 2}\, .
\end{equation}
Let us define $\mu$ by
\begin{equation}\label{eq:2.16}
  \mu=\sup \Big\{\ \Big|{c\over 2\pi i}F(s)\big(s^p+(-s)^p\big)\Big| \ :\ s\in [-i\,\!T,i\,\!T]\ \Big\}\, .
\end{equation}   
Then, since $s=i\,\!\omega$ there exists $\rho_0$ such that  
\begin{equation} \label{eq:2.17}
  \forall \rho\ge \rho_0\quad \big|1-e^{s^2/(2\rho^2)}\big|< {\epsilon\pi \over 2\mu T c}\, .
\end{equation}
Formulas (\ref{eq:2.12}), (\ref{eq:2.13}), (\ref{eq:2.15})-(\ref{eq:2.17}) imply that for any $x$ and for any $\rho\ge \rho_0$
\begin{align}
 & \big|f_p^e(x)-\phi_p^e(x,\rho)\big| \leq \big|{c\over 2\pi i}\big|  \big| \int_{-i\infty}^{i\infty} \big|\big|F(s)\big(s^p+(-s)^p\big)  e^{sx} \big|1-e^{s^2/(2\rho^2)}\big|    \big|ds\big|\, \notag \\
& \leq \Big|{c\over 2\pi i}\Big|\Big(\int_{i\,\!T}^{i\infty}+\int_{-i\infty}^{-i\,\!T}\Big)  \Big|F(s)\big(s^p+(-s)^p\big)\,ds\Big|+\Big|{c\over 2\pi i}\Big|\int_{-i\,\!T}^{i\,\!T}{\epsilon\pi \over 2\mu T c}\mu ds  \Big| \notag \\
& < {\epsilon \over 2}+{\epsilon \over 2} =\epsilon\, ,\notag
\end{align}
which proves (\ref{eq:2.14})
Therefore, $f_p^e(x)$ is continuous. The fact that $f_p^e(x)$ is $k$ times continuously differentiable follows from the fact that $m>p+1$ from (\ref{eq:2.8}), and the theorem on  differentiability of improper integrals by parameter (see [5]). This proves (i) for $\theta=e$. The proof for (ii) is entirely similar to the proof of (i) except that we need to multiply the functions under the integral by $s^k$, and we still use (\ref{eq:2.8}) to justify absolute convergence. Let us prove (iii).  Since $x\,\!f(x)$ is uniformly convergent, $F(s)$ is continuously differentiable in $s$. Therefore, $F(s)\big(s^p+(-s)^p\big)$ is also continuously differentiable in $s$. Absolute integrability of $F(s)\big(s^p+(-s)^p\big)$ again follows from (\ref{eq:2.8}). From (\ref{eq:2.13}) it follows that $f_p^e(-x)=\mathcal{F}^{-1}\big(c\,\!F(s)(s^p+(-s)^p)\big)$. Therefore, from the theorem on the inverse Fourier transform, it follows that $\mathcal{F}\big(f_p^e(x)\big)=c\big(s^p+(-s)^p\big)F(s)$. This proves (iii) for $\theta=e$. The proof for $\theta=o$ is entirely  similar.  
\end{proof}
Theorem 2.3 shows that as $\rho\to \infty$, the complexity of the smoothed signal increases and the uniform limit of the convolution $\phi_p^e(x,\rho)$ is the function $f_p^\theta(x)$. On the other hand, if $\rho\to 0$ (smoothing increases), the uniform limit also exists and is zero. This is shown by the following result.  
\begin{theorem}\label{thm2.4}
 If the conditions of Theorem 2.3 are satisfied, then \newline
\medskip
 
(i) Functions $\phi_p^e(x,\rho)$, $\theta=e,o$ uniformly converge as $\rho\to 0$ to an identically zero function.\newline
\medskip

(ii) Also, for every $k$, $0<k<m-p-1$, functions ${d^k\,\!\phi_p^e(x,\rho) \over dx^k}$, $\theta=e,o$ also converge uniformly to zero.   
\end{theorem}
\begin{proof}
Arguments as in the proof of Theorem~\ref{thm2.3} show that $\phi_p^e(x,\rho)$ is given by (\ref{eq:2.12}) and that the function $F(s)\big(s^p+(-s)^p\big)$ converges absolutely on  $(-i\infty,i\infty )$. Let $\epsilon>0$. Then from (\ref{eq:2.12}) and the fact that $e^{s^2/(2\rho^2)}<1$, it follows that $\exists T>0$ such that $\forall y>0$, we have 
\begin{equation}\label{eq:2.19}
\Big|{c\over 2\pi i}\Big|\Big(\int_{i\,\!T}^{i\infty}+\int_{-i\infty}^{-i\,\!T}\Big)  \big|F(s)e^{s^2/(2\rho^2)}e^{sx}\big(s^p+(-s)^p\big)\,ds\big| < {\epsilon \over 2}\, .
\end{equation}
Let $µ$ be defined by (\ref{eq:2.12}). Then, $\exists \rho_0>0$ such that for any $\rho<\rho_0$,
\begin{equation}\label{eq:2.20}
  \big|e^{s^2/(2\rho^2)} \big| < {\epsilon \over 4\mu T}\, .
\end{equation} 
Then (\ref{eq:2.12}), (\ref{eq:2.19}) and (\ref{eq:2.20}) imply that for $\rho<\rho_0$: 
\begin{align}
  & |\phi_p^e(x,\rho)|\ \leq\ \Big|{c\over 2\pi i}\Big(\int_{i\,\!T}^{i\infty}+\int_{-i\infty}^{-i\,\!T}\Big)K(s)\,ds\Big|+\Big|{c\over 2\pi i}\Big|\,\Big|\int_{-i\,\!T}^{i\,\!T}K(s)\,ds\Big|\ < \notag \\
& {\epsilon \over 2}\,+\,{\epsilon \over 4\mu T}\Big|{c\over 2\pi i}\Big|\,\Big|\int_{-i\,\!T}^{i\,\!T}\big|F(s)\big(s^p+(-s)^p\big)e^{sx}\big|\,ds\Big| \leq \notag \\
&{\epsilon \over 2}\,+\,\big({\epsilon \over 4\mu T}\big)\,2\mu T\ =\ \epsilon\, , \notag
\end{align}
where $K(s)=F(s)e^{s^2/(2\rho^2)}e^{sx}\big(s^p+(-s)^p\big)$. This proves (i) for $\theta=e$. If $\theta=o$ the proof is similar. The proof of (ii) is entirely similar also.  
\end{proof}
Theorems~\ref{thm2.3} and \ref{thm2.4} show the limiting values for $\phi_p^e(x,\rho)$ as $\rho\to \infty$ or $\rho\to 0$. The following theorem establishes the limit as $\rho\to \rho_0$ where $\rho_0$ is any positive number.   
\begin{theorem}\label{thm2.5}
If the conditions of Theorem~\ref{thm2.3} are satisfied, then  we have the  following convergence (uniformly)
$$
\lim_{y\to \rho_0}{d^k\,\!\phi_p^\theta(x,\rho) \over dx^k}= {d^k\,\!\phi_p^\theta(x,\rho_0) \over dx^k}\, ,
$$ 
where $0\leq k<m-p-1$, $\theta=e,o$. 
\end{theorem}
\begin{proof}
The proof is a simplified version of that of Theorem~\ref{thm2.3}, and is therefore omitted.   
\end{proof}
The following theorem is important in showing that functions $\phi_p^\theta(x,\rho)$ are essentially localized in a domain with bounded $x$ and $\rho$.  
\begin{theorem}\label{thm2.6}
Let the conditions of Theorem~\ref{thm2.3} be satisfied and let $\epsilon>0$. Then there exist $T>0$ and $Q>0$ such that if $|x|>T$ or if $\rho<Q$, then
\begin{align}
 |\phi_p^\theta(x,\rho)|&<\epsilon\, ,\notag \\
 {d^k\,\!\phi_p^\theta(x,\rho) \over dx^k}&<\epsilon\, .\label{eq:2.21}
\end{align}
\end{theorem}
\begin{proof}
 Since $m>p+1>0$, it follows from the conditions of our theorem that the function $f^\prime(x)$ is absolutely integrable. Using arguments similar to those of Lemma~\ref{lem2.1}, it is easy to show that the function $f^\prime(x)*\rho^{p+1}\Lambda_p^\theta(x)$ is absolutely integrable. This implies that for every $\rho$
\begin{equation}\label{eq:2.22}
   \lim_{|x|\to \infty}\phi_p^\theta(x,\rho)=0\, .
\end{equation}
Similarly, we can get  
\begin{equation} \label{eq:2.23}
    \lim_{|x|\to \infty}{d^k\,\!\phi_p^\theta(x,\rho) \over dx^k}=0\, \quad {\rm for\ }\ 0<k<m-p-2\, .
\end{equation}
Let $\epsilon>0$. From (\ref{eq:2.22}) and (\ref{eq:2.23}) and Theorems~\ref{thm2.3} and \ref{thm2.4}, it follows that there exist positive numbers $R$ and $Q$ such that
\begin{equation}\label{eq:2.24}
   \big|f_p^e(x)-\phi_p^\theta(x,\rho)\big| <{\epsilon \over 2}\quad \forall y>R,\ \forall x\, ,
\end{equation}
and
\begin{equation}\label{eq:2.25}
 \big| \phi_p^\theta(x,\rho)\big| < \epsilon \quad \forall y>R,\ \forall x,\ \forall y<Q\, . 
\end{equation} 
Since ${d\,\!f_p^\theta(x,\rho) \over dx}$ is absolutely integrable, $\lim_{|x|\to \infty}f_p^\theta(x,\rho)=0$. This implies that $\exists T_1$ such that 
\begin{equation}\label{eq:2.26}
  \big|f_p^\theta(x,\rho)\big| < {\epsilon \over 2}\quad \forall \forall |x|>T_1\, .
\end{equation}
Inequalities (\ref{eq:2.24}) and (\ref{eq:2.26}) imply that  
\begin{equation}\label{eq:2.27}
  \big| \phi_p^\theta(x,\rho')\big| < \epsilon \quad \forall y>R,\ \forall |x|>T_1\, . 
\end{equation}
Using (\ref{eq:2.22}) and Theorem~\ref{thm2.5}  we can show that for every $\rho\in [Q,R]$ there exists an open neighborhood $O_y$ such that for some positive number $T(\rho)$
\begin{equation}\label{eq:2.28}
  \big| \phi_p^\theta(x,\rho)\big| < \epsilon \quad \forall y' \in O_y\ {\rm and\ } |x|>T(\rho)\, . 
\end{equation}
Since $[Q,R]$ is compact, there exists a finite number of open sets $O_{y_1},\,\dots\,,O_{y_n}$, covering $[Q,R]$. Let $T=\max(\,T_l,\ T(y_i)\, ,\ i = 1,\,\dots\,,n)$, then (\ref{eq:2.25}), (\ref{eq:2.27}) and  (\ref{eq:2.28}) imply that  
$$
 \big| \phi_p^\theta(x,\rho)\big| < \epsilon\ {\rm if\ either\ }|x|>T\ {\rm or\ }y<Q\, .
$$ 
The second condition in (\ref{eq:2.21} is proved similarly.  
\end{proof}
\begin{remark} \label{rem2.5}
The results of this section can be generalized to functions $f(x)$ such that they and their derivatives have a finite number of jump discontinuities.
\end{remark}   
\begin{theorem}\label{thm2.7}
Let $m>p+1+2l$ with $l>0$ and let the conditions of Theorem~\ref{thm2.3} be satisfied. Then there is a function $\Psi_p^\theta(x,\sigma)$, $\theta=e,o$ defined for all real $\sigma$, $\sigma\in(-\infty,\infty)$ such that:
\begin{tabbing}
(iii) \= $\Psi_p^\theta(x,\sigma)=\Psi_p^\theta(x,-\sigma)$. \kill
(i)  \> $\Psi_p^\theta(x,\sigma)=\phi_p^\theta(x,\rho)$ for $\rho>0$, $\rho={1\over\sigma}$, and $\Psi_p^\theta(x,0)=f_p^\theta(x)$. \\
     \> \\
(ii) \> Function $\Psi_p^\theta(x,\sigma)$ has continuous mixed derivatives  ${\partial^{k+n}\over \partial x^k \partial \sigma^n}\Psi_p^\theta(x,\sigma)$ \\
     \> for any $k$  and $n$ such that  $2n+k\leq 2l$. \\
    \> \\
(iii)\> $\Psi_p^\theta(x,\sigma)=\Psi_p^\theta(x,-\sigma)$. \\
    \> \\
(iv) \> $\Psi_p^\theta(x,\sigma)\ \sim\ F(s)\big(s^p+(-s)^p\big)e^{s^2\sigma^2/2}\quad$ \= if $\theta=e$, \ and \\
     \> $\Psi_p^\theta(x,\sigma)\ \sim\ F(s)\big(s^p-(-s)^p\big)e^{s^2\sigma^2/2}\quad$ \> if $\theta=o$ \\
    \> \\
(v)  \>  $\sigma\,{\partial^2\,\Psi_p^\theta(x,\sigma)\over \partial x^2}$ for every $x$ and $\sigma$ \\
\end{tabbing}
\end{theorem}
\begin{proof}
Let $\theta=e$. Then we define $\Psi_p^\theta(x,\sigma)$ by
\begin{equation}\label{eq:2.29}
  \Psi_p^\theta(x,\sigma)={c\over 2\pi i}\int_{-i\infty}^{i\infty} F(s)\big(s^p+(-s)^p\big)e^{s^2\sigma^2/2}e^{sx}\,ds\, .
\end{equation}
where $F(s)=\mathcal{F}\left(f(x)\right)$, $s = i\,\!\omega$. As in the proof of Theorem~\ref{thm2.3}, one can show that  $F(s)\big(s^p+(-s)^p\big)e^{s^2\sigma^2/2}$  is indeed the Fourier transform of $\Psi_p^\theta(x,\sigma)$ for $\theta=e$. This proves (iv). (iii) is obvious. (i) follows from (\ref{eq:2.29}), Lemma~\ref{lem2.1}(b), and Theorem~\ref{thm2.3}. 

Let us prove  (ii). From the theorem on differentiation of improper integrals by a parameter (see [5], 11.55 a), we can interchange differentiation of the right hand side of (\ref{eq:2.30}) with integration if the derivative of the expression under the integral absolutely converges. 
However,
\begin{equation}\label{eq:2.30}
  {\partial^{k+n}\over \partial x^k \partial \sigma^n}\Big(F(s)\big(s^p+(-s)^p\big)e^{s^2\sigma^2/2}e^{sx} \Big)=s^kp(s,\sigma)e^{s^2\sigma^2/2}e^{sx}\big(F(s)\big(s^p+(-s)^p\big)\big)\,
\end{equation}
where $p(s,\sigma)$ is a polynomial whose degree on $s$ is less than or equal to $2n$.
 
Now (\ref{eq:2.30}) and (\ref{eq:2.30}) together with inequality $m>p+1+2l$ imply that if $2n+k\leq 2l$, the integral in (\ref{eq:2.29}) absolutely converges. This allows us to interchange integration and differentiation. Since the derivatives of all orders of the expression under the integral in (\ref{eq:2.29}) exist, (ii) follows. The proof for $\theta=o$ is entirely similar. In order to prove (v), we simply differentiate (\ref{eq:2.29}).   
\end{proof}
\begin{theorem}\label{thm2.8}
Let the conditions of Theorem~\ref{thm2.3} be satisfied  (in particular, $m>p+1$). Then if $\Delta>0$ is such that $p-\Delta> 0$, $p+\Delta+1<m$, then convergence in each of Theorems~\ref{thm2.3}, \ref{thm2.4} and \ref{thm2.5} is uniform with respect to $p_1\in[\,p-\Delta,\,p+\Delta\,]$. In other words, $\forall\epsilon>0$ $\exists \rho_0$ such that if $\rho>\rho_0$, then $|\phi_{p_1}^\theta(x,\rho)-f_{p_1}^\theta(x,\rho)|<\epsilon$, $\forall p_1 \in [\,p-\Delta,\,p+\Delta\,]$ (and similarly for convergences $\phi_{p_1}^\theta(x,\rho) \rightarrow \phi_{p_1}^\theta(x,\rho_1)$ as  $\rho\rightarrow y_1$ and $\phi_{p_1}^\theta(x,\rho) \rightarrow 0$ as $\rho \rightarrow0$ for $p_1 \in [\,p-\Delta,\,p+\Delta\,]$.   
\end{theorem}
\begin{proof}
The proof is entirely similar to the proofs of Theorems~\ref{thm2.3}, \ref{thm2.4} and \ref{thm2.5} except that we note that all inequalities can be satisfied uniformly on $[p-\Delta,p+\Delta]$
\end{proof}
\begin{theorem}\label{thm2.9}
  Let the conditions of Theorem~\ref{thm2.3} be satisfied and let $\epsilon>0$. Then $\exists T>0$ and $Q>O$ such that if $|x|>T$ or if $\rho<Q$, then $\forall p_1\in[p-\Delta,p+\Delta]$ where $p-\Delta>0$, $p+\Delta+l<m$ we have  
\begin{align}
& |\phi_{p_1}(x,\rho)|<\epsilon \label{eq:2.31}\\
\Big|{d^k\,\!\phi_{p_1}^\theta(x,\rho) \over dx^k}\Big|& < \epsilon \quad {\rm for\ }  0<k<m-p-2\, .\label{eq:2.32}
\end{align}
\end{theorem}
\begin{proof}
The result is derived from Theorem~\ref{thm2.8} in entirely the same way as the result of Theorem~\ref{thm2.6} is derived from Theorems~\ref{thm2.3}, \ref{thm2.4} and \ref{thm2.5}.
\end{proof}
\begin{theorem}\label{thm2.10}
Let the conditions of Theorem~\ref{thm2.3} be satisfied and let $M=[a,p]$ ($a>0$ if $\theta=o$ and $a\ge 0$ if $\theta=e$) be an interval on the real line. Then the set $D_\epsilon$ defined by condition
\begin{equation}\label{eq:2.33}
  D_\epsilon = \Big\{\ x,\ \sigma\ p_1\ :\  \big| \Psi_{p_2}^\theta(x,\sigma)\big|\ge \epsilon,\ p_1\in [0,p] \ \Big\}
\end{equation}
is compact if $\epsilon>0$.  
\end{theorem}
\begin{proof}
From Theorems~\ref{thm2.7} and \ref{thm2.9} it follows that $\forall \in M$, $\exists \Delta>0$ such that the set
\begin{equation}\label{eq:2.34}
    D_\epsilon(p_1) = \Big\{\ x,\ \sigma\ p_2\ :\ p_2\in[p_1-\Delta,p_1+\Delta],\ \big| \Psi_{p_1}^\theta(x,\sigma)\big|\ge \epsilon\ \Big\}
\end{equation}
is bounded if $\epsilon>0$. Since $M$ is compact, there is a finite covering of $M$ with intervals  $[\,q_j-\Delta_j,\, q_j+\Delta_j\,]$ such that the set of $D_\epsilon$ satisfying (\ref{eq:2.33}) for $p_1=q_j$ and $\Delta=\Delta_j$ is bounded. Each set $D_\epsilon(q_j)$ is bounded; thus, $D_\epsilon$ is also bounded as a union of a finite  number of bounded sets. Since $D_\epsilon$ is closed in $\mathbb{R}^3=\{x,\sigma,p\}$, it follows that $D_\epsilon$ is compact.  
\end{proof}
\begin{theorem}\label{thm2.11}
If in Theorem~\ref{thm2.10} $m>p+1+k$, then the set 
$$
 D_\epsilon^k = \Big\{\ x,\ \sigma\ p_1\ :\  \Big| {\partial^k\,\Psi_{p_2}^\theta\over dx^k}(x,\sigma)\Big|\ge \epsilon,\ p_1\in [0,p] \ \Big\}
$$ 
is compact if $\epsilon>0$.  
\end{theorem}
\begin{proof}
Let us take ${d^k\,f\over dx^k}$ instead of function $f(x)$ in Theorem~\ref{thm2.10}. Then the function defined by the right side of (\ref{eq:2.29}) where $F(s)=\mathcal{F}\Big({d^k\,f^\theta\over dx^k} \Big)$ is actually ${\partial^k\,\Psi_p^\theta\over dx^k}(x,\sigma)$. Thus, Theorem~\ref{thm2.11} follows from Theorem~\ref{thm2.10}.  
\end{proof}
\begin{theorem}\label{thm2.12}
(i) Consider functions $\Lambda_p^\theta(x)$ $\theta=e,o$ as functions of $p$ given by equalities (\ref{eq:1.3}) and (\ref{eq:1.4}) where $p$ is now an arbitrary complex number. Then $\Lambda_p^\theta(x$) is an analytic, even entire function of two variables $x$ and $p$.  
(ii) Let function $\Psi^\theta(p,x,\sigma)=\Psi_p^\theta(,x,\sigma)$ be defined by (\ref{eq:2.29}) where $F(s)$ is the  Fourier transform of $f(x)$. If $f(x)$ satisfies conditions of Theorem~\ref{thm2.3}, then $\Psi^\theta(p,x,\sigma)$ is  a continuous function of three variables. It has continuous mixed order derivatives ${\partial^{k+n+l}\,\Psi^\theta(p,x,\sigma)\over \partial x^k \partial \sigma^n \partial p^l}$ where $l$ is arbitrary, and $k$ and $n$ satisfy conditions of Theorem~\ref{thm2.7}.   
\end{theorem}
\begin{proof}
Let us assume $\theta=e$. The ratio test shows that the series (\ref{eq:1.3}) converges  absolutely and uniformly by $p$ and $x$ on every bounded set of pairs $(x,p)$, $|x|<R_1$, $|p|<R_2$ Since each term in (\ref{eq:1.3}) is an entire function, the sum is also entire. This proves (i).  The proof of (ii) is similar to that of (ii) of Theorem~\ref{thm2.7} and is therefore omitted.  
\end{proof}
\begin{remark}\label{rem2.6}
Since $\phi_p^\theta(x,\rho)=\rho^{p+1}\Lambda_p^\theta(x\rho)*f(x)$, it is easy to show that if the $i$ th  derivatives of $f(x)$ are absolutely convergent for $i\leq m$, then ${\partial^i\,\phi_p^\theta\over dx^i}(x,\rho)=\rho^{p+1}\Lambda_p^\theta(x\rho)*{\partial^i\,f(x)\over dx^i}$ and therefore  ${\partial^i\,\Psi_p^\theta\over dx^i}(x,\sigma)={1\over \rho^{p+1}}\Lambda_p^\theta({x\over \sigma})*{\partial^i\,f(x)\over dx^i}$ for $\sigma>0$. This shows that if a certain statement is true about contours $\Psi_p^\theta(x,\sigma)=c$, then it is true for contours ${\partial^i\,\Psi_p^\theta(x,\sigma)\over dx^i}=c$ if the assumptions for ${\partial^i\,f(x)\over dx^i}$ in the latter case are the same as the assumptions for $f(x)$ in the former. 
\end{remark}
\section{Topological Properties of Level Lines of the Convolutions}
  \label{sec3}
For this section we will assume that the conditions of Theorem~\ref{thm2.3} are satisfied. We will use some concepts of differential topology which can be found in [7]. We assume that $\pi$ and $p$ are fixed, and the notation $\Psi(x,\sigma)$ will be used instead of $\Psi_p^\theta(x,\sigma)$. Also, ${\partial^{m+n}\,\Psi(x,\sigma)\over \partial x^m \partial \sigma^n }$  will be denoted by $\Psi_{mxn\sigma}$ Let
\begin{equation}\label{eq:3.1}
  M_c=\{\,X,\Sigma\ :\ \Psi(x,\sigma)=C \}\, .
\end{equation}
Assume that $M_c$ does not contain points such that $\Psi_x(x,\sigma)=\Psi_c(x,\sigma)$. Then, from  the implicit function theorem, it follows that $M_c$ is a one-dimensional manifold. Since $M_c$ is  compact by Theorem~\ref{thm2.10}, $M_c$ is a union of a finite number of closed curves (curves homeomorphic to a circle). Since $\Psi(x,\sigma)=\Psi(x,-\sigma)$, $M_cc$ must be symmetric with  respect to the $x$-axis. Let us prove that each component $K$ of $M_cc$ intersects $x$ - axis.  Assume the contrary. Let $K$ lie entirely in the upper half plane $\sigma>0$. Then, since $K$ is compact, there exists a point $(x_0,\sigma_0)\in K$ such that its $\sigma$-coordinate reaches absolute minimum on $K$. By assumption, either $\Psi_x(x_0,\sigma_0)\neq 0$ or $\Psi_\sigma(x_0,\sigma_0)\neq 0$ or both. Assume $\Psi_x(x_0,\sigma_0)\neq 0$; then, by the implicit function theorem there exists a neighborhood $V$ of $\sigma_0$, and a function $x(\sigma)$ defined on $V$ and such that $\Psi(x(\sigma),\sigma) = c$,  $x(\sigma_0)=x_0$. This shows that there are points in $K$ having smaller $\sigma$-coordinates than $\sigma_0$.  Thus, we come to a contradiction and therefore  $\Psi_x(x_0,\sigma_0)=0$. This implies that
\begin{equation}\label{eq:3.2}
  \Psi_\sigma(x_0,\sigma_0\neq 0\, .
\end{equation}  
Therefore, by the implicit function theorem, $\exists$ a function of $\sigma(x)$ in a neighborhood of $x_0$, such that $\Psi(x,\sigma(x))=c$, $\sigma(x_0)=\sigma_0$. Since $\sigma$ achieves a minimum at $x_0$, $\sigma^\prime(x_0)=0$. Let us differentiate equality $\Psi(x,\sigma(x))=c$ by $x$. Then,  $\Psi_x+\Psi_\sigma\sigma^\prime(x)=0$. Differentiating this again we get:
\begin{equation}
 \Psi_{xx}+ \Psi_{x\sigma}\sigma^\prime(x)+\Psi_{\sigma\sigma}(\sigma^\prime(x))^2+\Psi_{\sigma x}\sigma^\prime(x)+\Psi_\sigma\sigma^{\prime\prime}(x)=0\, .\notag
\end{equation}
Substituting $x=x_0$ and using equality $\sigma^\prime(x_0)=0$, we get $\Psi_{xx}+\Psi_\sigma\sigma^{\prime\prime}(x_0)=0$. Now using (v) of Theorem~\ref{thm2.7}, we get $\Psi_\sigma\sigma^{\prime\prime}(x_0)=-{\Psi_\sigma \over \sigma_0}$. Since $\Psi_\sigma(x_0,\sigma_0)\neq 0$ this implies that $\sigma^{\prime\prime}(x_0)=-{1 \over \sigma_0}<0$. This, however, contradicts the fact that $\sigma(x_0)$ reaches a minimum at $x_0$. Thus, we again arrive at a contradiction, and $K$ intersects the $x$-axis. This shows that $K$ is itself symmetric with respect to the $x$-axis.  
\begin{theorem}\label{thm3.1}
Under generic assumptions (such that there is no point $(x,\sigma)$ satisfying 
${\partial^{j+1}\,\Psi_p^\theta(x,\sigma)\over \partial x^j \partial \sigma }= {\partial^{j+1}\,\Psi_p^\theta(x,\sigma)\over \partial x^{j+1}}=0$ and ${\partial^{j}\,\Psi_p^\theta(x,\sigma)\over \partial x^{j}}=c$ the set $M_c^{j\,!p} = \{\,x,\sigma\ :\ {\partial^{j}\,\Psi_p^\theta(x,\sigma)\over \partial x^{j}}= c\,\}$ where $c\neq 0$ and $p\ge 0$ is a finite union of closed smooth curves, and each one of them is symmetric with respect to the $x$-axis.   
\end{theorem}
\begin{proof}
The statement for $i=0$ has been proved in the above argument. The case of $i>0$ follows from Remark~\ref{rem2.6}.  
\end{proof}
Now that we have determined the topological structure of the level lines of functions $\Psi_{k\,\!x}(x,\sigma)$ for $c\neq 0$ (under some generic assumptions), we want to investigate the structure of lines $\Psi_{k\,\!x}(x,\sigma)=0$ (which is the same as ${\partial^k\,\Psi_p^\theta(x,\sigma)\over \partial x^k}=0$ in our notations). For this we need some preliminary results.
\begin{lemma}\label{lem3.1}
 Let $V(x)$ be a smooth vector field on a Euclidean space $\mathbb{R}^n$ and let $z(t)$, $t\in (-\infty,\infty)$ be a trajectory of $V(x)$. Let $X$ be the union of all points $z(t)$, $t\in (-\infty,\infty)$ and let the closure $\overline{X}$ of $X$ be a one-dimensional sub-manifold of  $\mathbb{R}^n$. Assume also that  $\overline{X}$ does not contain equilibrium points of the vector field $V(x)$. Then, $X$ is a closed set, i.e. $X=\overline{X}$ .  
\end{lemma}
\begin{proof}
Assume on the contrary that $\overline{X}\setminus X \neq \emptyset$, i.e., that $X$ is not closed. Let  $\rho\in \overline{X}\setminus X $. Since $\overline{X}\setminus X$ is invariant with respect to $V(x)$, i.e., it contains the whole trajectory containing $\rho$ (see [8], proposition 1.4), $\overline{X}\setminus X$ contains the trajectory of $\rho$, $\rho(t)$ such that $\rho(0)=y$. If the trajectory of $\rho$ is not a single point, then since $\overline{X}$ is a one-dimensional sub-manifold of $\mathbb{R}^n$, there is an $\epsilon>0$ and an open neighborhood $W$, $\rho\in W \subset \mathbb{R}^n$ such that  
\begin{equation}\label{eq:3.3}
  W\cap \overline{X} = \cup\{\,\rho(t)\ :\ -\epsilon<t<\epsilon\,\}\, .
\end{equation}
Since $\rho(t)\in \overline{X}\setminus X$, this implies that $W\cap X=\emptyset$ and $\rho\in W$. This, however, contradicts the assumption that $\rho$ is a limit point of $X$. Thus, the trajectory of $\rho$ consists of a single point. But then obviously, the vector field $V$ is zero at $\rho$, and thus $\rho$ is an equilibrium point of $V$. This again contradicts the conditions of the Lemma. Therefore, $X$ is a closed set.  
\end{proof}
We will also need the following: 
\begin{definition}\label{def3.1}
A vector field $V$ on a manifold $M$ is called $\omega$-complete if any  trajectory of $V$ defined on the interval  $[t_0,t_1]$ can be extended to $[t_0,\infty]$. It is called  $\alpha$-complete if every trajectory defined on $[t_0,t_l]$ can be extended to  $(-\infty,t_1]$. $V$ will  be called \textit{complete} if it is both $\alpha$- and $\omega$-complete.  
\end{definition}
A vector field $V$ satisfying Lipschitz conditions with a uniform constant $K$, i.e., $\|V(x)-V(y)\|< K\|x-y\|$ for all $x$ and $\rho$, is complete [10]. In particular, if the  derivatives of the components of vector $V(x)$ are uniformly bounded, then the Lipschitz condition is satisfied and $V$ is complete. In order to tie our discussion about vector fields with  the level contours of functions ${\partial^j\,\Psi_p^\theta(x,\sigma)\over \partial x^j}$, need to introduce a vector field $V$ whose trajectories are such level contours. 

Obviously, we just need to consider a vector field which is perpendicular to the gradient of the function ${\partial^j\,\Psi_p^\theta(x,\sigma)\over \partial x^j}$,  $j=0,1,2,\dots$. Let $\Psi(x,\sigma)=\Psi_p^\theta(x,\sigma)$ as before. It is convenient for us to consider the level contours of $\Psi_x$. (According to Remark~\ref{rem2.6}, this is a general case.) Define vector field $V$ on the plane $(x,\sigma)$ as $(\Psi_{x\,\sigma},\Psi_{x\,x})$. Then, we have the system of differential equations of second order:
\begin{align}
  \dot{x} & = -\Psi_{x\,\sigma}(x,\sigma) \label{eq:3.4} \\
\dot{\sigma} & = -\Psi_{x\,x}(x,\sigma)\, ,\notag
\end{align}
such that $\Psi_x$ is constant on its trajectories. (Indeed, $V$ is perpendicular to the gradient $\nabla \Psi_x(x,\sigma)$. We will need the following very useful result, which shows that there is an energy function defined for the vector field $V(x,\sigma)$.
\begin{lemma}\label{lem3.2}
Let $T(t)$ be any trajectory of the vector field $V(x,\sigma)$ defined by (\ref{eq:3.4}). Then function $L(x,\sigma)=\Psi(x,\sigma)-x\Psi_x(x,\sigma)$ is non-decreasing for $\sigma>0$ on $T(t)$ and non-increasing for $\sigma<0$. Generically, it is strictly increasing on every interval $(t_1,t_2)$ of the argument of the trajectory $T(t)$ for $\sigma>0$, and strictly decreasing for $\sigma<0$.  
\end{lemma}
\begin{proof}
Along every trajectory $T(t)$,  
\begin{align}
  {d\,\!L(T(t))\over dt} &= \langle V(x,\sigma), \nabla L(x,\sigma) \rangle \notag \\
& = -\Psi_{x\,\!\sigma}(\Psi_x-x\,\!\Psi_{x\,\!x}-\Psi_x)+\Psi_{x\,\!x}(\Psi_\sigma-x\,\!\Psi_{x\,\!\sigma}) \notag \\
& = \Psi_{x\,\!x}\Psi_\sigma=\sigma \Psi_{x\,\!x}^2 \notag
\end{align}
(the last equality follows from (v), Theorem~\ref{thm2.7}). This shows that $L(T(t))$ is non-decreasing for  $\sigma \ge 0$ and is non-increasing for $\sigma<0$. Since generically, $\Psi_{x\,\!x}$ is not identically zero on any curve $\Psi_x(x,\sigma)=c$ , our result follows.  
\end{proof}
\begin{lemma}\label{lem3.3}
If the conditions of Theorem~\ref{thm2.7} are satisfied, then functions ${\partial^{k+n}\,\Psi_p^\theta(x,\sigma)\over \partial x^k\partial\sigma^n}$, are uniformly bounded on the plane $(x,\sigma)$ for $2n+k<-2l$.  
\end{lemma} 
\begin{proof}
According to (\ref{eq:2.29}) for $\theta=e$, for some constant $c$,
\begin{equation}\label{eq:3.5}
  \Psi(x,\sigma)=\Psi_p^\theta(x,\sigma)={c_1\over 2\pi}\,\!\int_{-\infty}^{\infty}F(i\,\!\omega)|\omega|^p\,\!e^{-\omega^2\sigma^2/2}e^{i\,\!\omega\,\!x}\,d\omega
\end{equation}
and from the proof of Theorem~\ref{thm2.7}, it follows that we can differentiate under the integral. 
Thus,  
\begin{equation}\label{eq:3.6}
  \Psi_{k\,!x\,!n\,!\sigma}(x,\sigma)={c_1\over 2\pi}\,\!\int_{-\infty}^{\infty}P(\omega,\sigma)F(i\,\!\omega)|\omega|^p\,\!e^{-\omega^2\sigma^2/2}e^{i\,\!\omega\,\!x}\,d\omega
\end{equation}
where $P(\omega,\sigma)$ is a polynomial of $\sigma$ and $\omega$ whose degree on $\sigma$ is $n$ and on $\omega$ is $2n+k$. Let now $\sigma>T$ where $T$ is an arbitrary number greater than zero. Then clearly, $\exists\,M>0$ such that 
\begin{equation}\label{eq:3.7}
  \sup\Big\{\big| P(\omega,\sigma)e^{-\omega^2\sigma^2/2}\big|\ :\ \sigma>T, \omega\in(-\infty,\infty)\,\Big\}<M\, .
\end{equation}
From (\ref{eq:2.8}) it follows that for $\\omega|\geq 1$
\begin{equation}\label{eq:3.8}
  |F(i\,\!\omega)|\,\!|\omega|^p|e^{i\,\!\omega\,\!x}|<{c\,\!|\omega|^p\over |\omega|^m} = c\,\!|\omega|^{p-m}<c\,\!|\omega|^{-2l-1}\, .
\end{equation}
Also, if $|F(i\,\!\omega)|<c_2$ for $|\omega|\leq 1$, then  
\begin{equation}\label{eq:3.9}
   |F(i\,\!\omega)|\,\!|\omega|^p|e^{i\,\!\omega\,\!x}|<c_2\quad\ {\rm for\ }|\omega|\leq 1\, .
\end{equation}

Now from (\ref{eq:3.5}) - (\ref{eq:3.9}) it follows that for $\sigma>T$ and any $x$  
\begin{equation}\label{eq:3.10}
  |\Psi_{k\,!x\,!n\,!\sigma}(x,\sigma)| \leq {|c_1\,\!c_2|\,\!M \over 2\pi}+\big|c_1 \over 2\pi \big|\,\!2M \int_1^\infty {|c| \over |w|^{2l+1}}\,d\omega\, .
\end{equation}
Obviously since $l>0$, the right side of inequality (\ref{eq:3.9}) is a real number denoted by $D_1$. Thus, we have:   
\begin{equation}\label{eq:3.11}
  |\Psi_{k\,!x\,!n\,!\sigma}(x,\sigma)|\leq D_1\quad \forall\,\sigma>T\quad\ {\rm and\ any\ }x\, .
\end{equation}
Let now, $0\leq\sigma\leq TT$. Then, since $P(\omega,\sigma)$ is a polynomial of degree $(2n+k)$ on $\omega$,  there exist constants $\alpha$ and $\beta$ such that  
\begin{equation}\label{eq:3.12}
 |P(\omega,\sigma)|\leq \alpha|\omega|^{2n+k}+\beta \qquad {\rm for\ }0<\sigma<T\quad {and\ any\ }\omega\, .
\end{equation}
Then for $0\leq \sigma\leq T$
\begin{equation}\label{eq:3.13}
  \big|P(\omega,\sigma)\,\! F(i\,\!\omega)\,\!\omega^p\,\!e^{i\,\!\omega\,\!x}\big|\leq (\alpha|\omega|^{2n+k}+\beta)|\omega|^p\,\!|F(i\,\!\omega)|\, . 
\end{equation}
Clearly, there exists $M_2>0$ such that 
\begin{equation}\label{eq:3.14}
  \big|P(\omega,\sigma)\,\! F(i\,\!\omega)\,\!\omega^p\,\!e^{i\,\!\omega\,\!x}\big|<M_2\quad {\rm for\ }|\omega|<1\, .
\end{equation}
Also, from (\ref{eq:2.8}) and the fact that $md-2n-k-p>m-21-p>1$, it follows that the integral $\int_1^\infty(\alpha|\omega|^{2n+k}+\beta)|\omega|^p\,\!|F(i\,\!\omega)|\,d\omega$ converges. Let 
\begin{equation}\label{eq:3.15}
  \int_1^\infty(\alpha|\omega|^{2n+k}+\beta)|\omega|^p\,\!|F(i\,\!\omega)|\,d\omega=M_3,\ .
\end{equation}
Then, from (\ref{eq:3.9}), (\ref{eq:3.9})), (\ref{eq:3.9})), (\ref{eq:3.9})) and (\ref{eq:3.9})), it follows that
\begin{align}
&|\Psi_{k\,!x\,!n\,!\sigma}(x,\sigma)|=\Big|{c_1\over 2\pi}\,\!\int_{-\infty}^{\infty}P(\omega,\sigma)F(i\,\!\omega)|\omega|^p\,\!e^{-\omega^2\sigma^2/2}e^{i\,\!\omega\,\!x}\,d\omega\Big| \leq   \label{eq:3.16} \\
&2\big|{c_1\over 2\pi}\big|\,\!M_2+2\big|{c_1\over 2\pi}\big|\,\!M_3=D_2\quad {\rm for\ }0\leq \sigma \leq T\ \ {\rm and\ any\ }x\, .
\end{align}
Finally, (\ref{eq:3.11}) and (\ref{eq:3.16}) prove the Lemma for any $\sigma>0$. Since $\Psi_{k\,!x\,!n\,!\sigma}(x,-\sigma)$, the Lemma is true for any $\sigma$ and $x$.  
\end{proof}
\begin{lemma}\label{lem3.4}
If the conditions of Theorem~\ref{thm2.7} are satisfied for $k+2<2l$, then the following vector field is complete:  
\begin{align}
  \dot{x} & = -\Psi_{k\,\!x\,\!\sigma}(x,\sigma) \label{eq:3.17} \\
\dot{\sigma} & = -\Psi_{(k+1)\,\!x}(x,\sigma)\, ,\notag
\end{align} 
In particular, vector field (\ref{eq:3.4}) is complete.   
\end{lemma}
\begin{proof}
As we mentioned earlier, the existence of a global Lipschitz constant guarantees completeness. Since the third order partial derivatives of $\Psi$ are bounded by Lemma~\ref{lem3.3}, there exists such a constant for the vector field (\ref{eq:3.17}) and our result follows.  
\end{proof}

Now we are ready to investigate the structure of curves 
$$
\{\,x,\sigma\ :\ {\partial^k\,\Psi_p^\theta(x,\sigma)\over \partial x^k} = 0\,\}, \quad k=0,1,\dots
$$
\begin{theorem}\label{thm3.2}
Under some generic assumptions (such that there is no points $(x,\sigma)$ satisfying
${\partial^k\,\Psi_p^\theta(x,\sigma)\over \partial x^k} = 0$, ${\partial^{k+1}\,\Psi_p^\theta(x,\sigma)\over \partial x^{k+1}}={\partial^{k+1}\,\Psi_p^\theta(x,\sigma)\over \partial x^k \partial\sigma}=0$  the set 
$$
M_0^{kp}\{\,x,\sigma\ :\ {\partial^k\,\Psi_p^\theta(x,\sigma)\over \partial x^k} = 0\,\}
$$
is a union of:  
(i) closed curves $C_j$ smooth and symmetric with respect to the the $x$-axis (and intersecting it at exactly two points);
(ii) curves $D_j$ diffeomorphic to the real line such that each $D_j$ is symmetric with respect to the $x$-axis and intersects it at exactly one point. 

\noindent (We assume that the conditions of Theorem~\ref{thm2.5} are satisfied and that $p\ge 0$.)  
\end{theorem}
\begin{remark}\label{rem3.1}
In Theorem~\ref{thm3.1} and in (i) of Theorem~\ref{thm3.2}, the fact that a closed curve is symmetric w.r.t. the $x$-axis trivially implies that this curve intersects the $x$-axis in exactly two points. Also, it is obvious that a curve diffeomorphic to the real line cannot intersect $x$-axis in more than one point if it is symmetric with respect to the $x$-axis. (Otherwise it would contain a closed curve). 
\end{remark}
\begin{proof}[Proof of Theorem~\ref{thm3.2}]
From the assumed generic conditions and the implicit function theorem, it follows that $M_0^{kp}$ is a one-dimensional sub-manifold of the $(x,\sigma)$ plane. Therefore, its connected components are either closed curves denoted by $C_j$ or curves $D_j$ diffeomorphic to the real line. Just as in Theorem~\ref{thm3.1}, one can prove that each curve $C_j$ is symmetric with respect to the $x$-axis (note that the symmetry of $M_0^{kp}$ about the $x$-axis does not immediately imply the symmetry of each of its components). This shows that (i) is satisfied. Let us prove (ii). Let $D_k$ be a component of $M_0^{kp}$ which is diffeomorphic to the real line. Let $x_0\in D_k$ be any point. Consider a trajectory $z(t)$ through $x_0$ of the vector field $V$ defined by (\ref{eq:3.17}). Since $V$ is complete by Lemma~\ref{lem3.4}, $z(t)$ is defined for all $t\in (-\infty,\infty)$. Also, since $\langle V(x,\sigma)\\nabla F(x,\sigma) \rangle = 0$ for every $x,\sigma$, where $F(x,\sigma)={\partial^{k}\,\Psi_p^\theta(x,\sigma)\over \partial x^{k}}$, the trajectory $z(t)$ is contained in $D_k$. Let $X$ be the union of all points of the trajectory $z(t)$, and let $\overline{X}$ be its closure. Obviously $\overline{X}\subset D_k$. Since $M_0^{kp}$ does not contain equilibrium points of the vector field $V$ and is a one-dimensional manifold, it follows from the tubular flow theorem [8], that any invariant set in $M_0^{kp}$  is open in $M_0^{kp}$ also. 

Since $\overline{X}$ is invariant, it is open in $M_0^{kp}$ and is, therefore, open in $D_k$ since $D_k$ is open in $M_0^{kp}$. Since $\overline{X}$ is open and closed in $D_k$, it coincides with $D_k$ because $D_k$ is connected.  From Lemma~\ref{lem3.1} it follows that $X$ is closed. This implies that $lim_{t\to\infty}|z(t)|=lim_{t\to-\infty}|z(t)|=\infty$. Suppose $X=D_k$ does not intersect $x$-axis, then, the $\sigma$-coordinate of $z(t)$ preserves the sign. Assume it is always positive, i.e., the $\sigma$-coordinate of $z(t)$ is greater than zero for every $t$. In Lemma~\ref{lem3.2}, let $\Psi(x,\sigma)={\partial^{k-1}\,\Psi_p^\theta(x,\sigma)\over \partial x^{k-1}}$, then the function 
$$
L(x,\sigma)=\Psi(x,\sigma)-x\,\!\Psi_x(x,\sigma)={\partial^{k-1}\,\Psi_p^\theta(x,\sigma)\over \partial x^{k-1}}-x\,\!{\partial^{k}\,\Psi_p^\theta(x,\sigma)\over \partial x^{k}}={\partial^{k-1}\,\Psi_p^\theta(x,\sigma)\over \partial x^{k-1}}
$$
is non- decreasing on the trajectory $X$ as $T\to\infty$ Let $(x,\sigma)=z(t)$ be an arbitrary point on $X$ such that ${\partial^{k-1}\,\Psi_p^\theta(x,\sigma)\over \partial x^{k-1}}=c \neq 0$. For example, let $c > 0$. Then, $c_1=\lim_{t\to\infty}{\partial^{k-1}\,\Psi_p^\theta(x,\sigma)\over \partial x^{k-1}}\geq c>0$. However, since $\lim_{t\to\infty}|z(t)|=\infty$, we should have $c_1=0$ according to Theorem~\ref{thm2.11}. Indeed, the set of all points $(x,\sigma)$ such that ${\partial^{k-1}\,\Psi_p^\theta(x,\sigma)\over \partial x^{k-1}}\geq c$ is bounded on $x,\sigma$ by Theorem~\ref{thm2.11}. However, this inequality is satisfied for all $t>t_0$ and the set $\cup\{\,z(t), t\geq t_0\,\}$ is unbounded since $|z(t)|\to\infty$ as $t\to \infty$. 
Similarly, we come to a contradiction assuming $c<0$. Thus, $X=D_k$ does intersect the $x$-axis. This can happen at only one point according to Remark~\ref{rem3.1}.   
\end{proof}
Now we have established all the theoretical background necessary to construct the trees for the contours ${\partial^{k}\,\Psi_p^\theta(x,\sigma)\over \partial x^{k}}=c$. In the case where $c=0$, $\theta=e, p=0$, these are zero crossings of the Gaussian convolved with ${d^k\,\!f \over dx^k(x)}$.  The connections of our results with fractional calculus seem to be very promising and interesting. In fact, we obtain the whole continuum of trees which give us additional means to characterize signals. The method of convolution with a Gaussian has been applied successfully in a number of papers to edge detection, vision and pattern recognition. See, for example [11] and [12]. Notice that if instead of taking convolutions of $f(x)$ with $\rho^{p+1}\Lambda_p^\theta(x\rho)$, $\theta=e,o$, we take convolutions with $\Lambda_{\alpha\,\!\beta\,\!p}(x)=\big(\alpha\Lambda_p^e(x)+\beta\Lambda_p^o(x)\big)\rho^{p+1}$  where $\alpha,\beta$ are arbitrary parameters, we obtain a two-parametric family of kernels, and all the results obtained in this section are still true for the functions defined via kernels $\Lambda_{\alpha\,\!\beta\,\!p}(x)$.  
\begin{theorem}\label{thm3.3}
The results of Theorems~\ref{thm2.2}-\ref{thm2.12}  and \ref{thm3.1}-\ref{thm3.2} may be generalized for the convolutions based on kernels $\Lambda_{\alpha\,\!\beta\,\!p}(x)$.  
\end{theorem}
The proofs of Theorems Theorems~\ref{thm2.2}-\ref{thm2.12}  and \ref{thm3.1}-\ref{thm3.2} can be only slightly modified to obtain the results for  $\Lambda_{\alpha\,\!\beta\,\!p}(x)$. Theorem~\ref{thm3.3} allows further generalization.  

Assume that $f(x)$ satisfies the conditions of Theorem~\ref{thm2.3}. Assume also that $g(x,\rho)$, $\rho>0$ is infinitely differentiable on $x$, and let $\mu(s,p)$ be a family of functions continuous in both variables, $s =i\,\!\omega$, $p\ge 0$ and suppose that the following conditions are satisfied:  
(1) $g(x,\rho)$ and $G(s,\rho)$ form a Fourier pair;\newline
\medskip
 
(2) For any $\rho>0$ and $p\ge0$, the product $G(s,\rho)\cdot \mu(s,p)$ is absolutely integrable on $(-i\,\!\infty,i\,\!\infty)$ as a function of $s=i\,\!\omega$; \newline
\medskip

(3) On any finite interval $(-i\,\!T,i\,\!T)$, $G(s,\rho)\to G(s,\rho_0)$ uniformly as $\rho\to \rho_0$. Limits also exist and are defined for $\rho=0$ and $\rho=\infty$, and are denoted by $G(s,0)$ and $G(s,\infty)$. The latter two functions are smooth and absolutely integrable as functions of $s=i\,\!\omega$. \newline
\medskip

(4) If we replace $\rho$ with ${1 \over \sigma}$, $G_1(s,\sigma)=G(s,\rho)$, $G_1(s,0)=G(s,\infty)$, then $G_1(s,\sigma)$ has continuous partial derivatives in $s$ and $\sigma$ at $\sigma=0$.\newline
\medskip 

Define $g(x,\rho,p)$ as 
\begin{equation}\label{eq:3.18}
  g(x,\rho,p)={1\over 2\pi\,\!i}\,\int_{-\infty}^{\infty}G(s,\rho)\mu(s,p)e^{s\,\!x}\, ds\, .
\end{equation}
Then, with the above assumptions, almost without change in the proofs we can show that: 
\begin{theorem}\label{thm3.4}
If $\Psi_G(x,\rho,p)$ is defined by:
\begin{equation}
  \Psi_G(x,\rho,p)=f(x)*g(x,\rho,p)\, ,\notag
\end{equation}
then
(i) For any $\rho>0$ we have uniform convergence in $x$:    
$$
\lim_{y\to \rho_0}\Psi_G(x,\rho,p)=\Psi_G(x,\rho_0,p)\, .
$$
(ii) There are functions $\Psi_G(x,0,p)$ and $\Psi_G(x,\infty,p)$ such that uniformly in $x$ we  have  
$$
\lim_{y\to \infty}\Psi_G(x,\rho,p)=\Psi_G(x,\infty,p) \qquad \lim_{y\to \infty}\Psi_G(x,\rho,p)=\Psi_G(x,0,p),\ .
$$
(iii) Let function $f(x)$ satisfy the conditions of Theorem~\ref{thm2.7}. If we define $\Psi_G^1(x,\sigma,p)$ as  $\Psi_G(x,{1\over y},p)$ for $\sigma>0$ and $\Psi_G^1(x,\sigma,p) = \Psi_G(x,\infty,p)$ for $\sigma=0$ and  $\Psi_G^1(x,\sigma,p)=\Psi_G^1(x,-\sigma,p)$ for $\sigma<0$, then $\Psi_G^1(x,\sigma,p)$ is a correctly  defined function having partial derivatives of orders $k$ in $x$ and $n$ in $\sigma$  where  $2n+k \le 2l$ ($l$ is defined in Theorem~\ref{thm2.7}.\newline
\medskip

iv) We also have the following Fourier pairs:  
$$ 
\Psi_G^1(x,\sigma,p)\ \sim\ \mu(s,p)F(s)G_1(s,\sigma)\quad p\ge 0
$$
(where for $\sigma<0$, $G_1(s,\sigma)$ is defined as $G_1(s,-\sigma)$).\newline
\medskip
 
(v) Convergence in (i) and (ii) is uniform in $p$ on the segment $p\in [0,p_0]$,  where $p_0$ is such that the conditions of Theorem~\ref{thm2.3} are satisfied for $p=p_0$.\newline
\medskip
 
(vi) The set $M$ defined by
$$          
M=\{\,x,\sigma,p\ :\ {\partial^k \over \partial x^k}\,\Psi_G^1(x,\sigma,p)=c,\ c=0,\ p \in [0,p_0]\,\}
$$
where $k+p+1<m$ and $p_0$ is, as in (v), compact.  
\end{theorem}
Thus far, we have considered cases when $\mu(s,p)=s^p+(-s)^p$ or $\mu(s,p)=s^p-(-s)^p$ and $G(s,\rho)=e^{s^2\over 2\rho^2}$\, .  

\section{Bifurcations of the Scale Space Contours}
  \label{sec4}
In this section, we are going to apply theory developed in Sections~\ref{sec2} and \ref{sec3} to the  analysis of bifurcations of contours as $p$ changes. We will rely especially on compactness  results, such as Theorems~\ref{thm2.10}, \ref{thm2.11}, \ref{thm3.1}, \ref{thm3.2} and \ref{thm3.3}. The main reason for this is that the  Morse theory which we use for bifurcation points analysis is not applicable in general to non-compact manifolds. As we have shown for every $\alpha$, $\beta$ and $p$, one can define kernels  $\alpha\Lambda_p^e(x\rho)+\beta\Lambda_p^o(x\rho)$, and if $f(x)$ is a function satisfying certain conditions, then one can define a convolution  
\begin{equation}\label{eq:4.1}
  \rho^{p+1}\big(\alpha\Lambda_p^e(x\rho)+\beta\Lambda_p^o(x\rho) \big)*f(x)=\phi_{\alpha\,\!\beta\,\!\pi}(x,\rho)\, .
\end{equation}
Proceeding as in Theorem~\ref{thm2.7}, one can define a function $\Psi_{\alpha\,\!\beta\,\!\pi}(x,\rho)$ from the function $\phi_{\alpha\,\!\beta\,\!\pi}(x,\rho)$. Now, if we fix an arbitrary $c$ and consider the equation $\Psi_{\alpha\,\!\beta\,\!\pi}(x,\sigma)=c$ for fixed $\alpha$, $\beta$, $p$, we obtain the contour from which we can construct the tree as in [3]. Function $\Psi_{\alpha\,\!\beta\,\!\pi}(x,\sigma)$ can be defined for arbitrary $\alpha$, $\beta$, $p$, $x$, and $\sigma$ real or complex.  If $\alpha$, $\beta$, $p$, $x$, $\sigma$, and $c$ are complex, then the equation  
\begin{equation}\label{eq:4.2}
  \Psi_{\alpha\,\!\beta\,\!\pi}(x,\sigma)=c
\end{equation}
defines an 8-dimensional surface in a 10-dimensional Euclidean space. In the real case it is a  4-dimensional surface in a 5-dimensional Euclidean space. It is clear that if $\alpha_l=\lambda\,\!\alpha$, $\beta_1=\lambda\,\!\beta$, then the equation $\Psi_{\alpha_1\,\!\beta_1\,\!\pi}(x,\sigma)={c \over \lambda}$ defines the same curve as (\ref{eq:4.2}). Now let us fix real numbers $\alpha$, $\beta$, $p$ ,$c$ and construct trees $TT$ and $TW$ for the curves  determined by (\ref{eq:4.2}) as explained in [3]. Clearly, if for a different set of parameters, trees $TT$ are equivalent, then the corresponding curves are topologically equivalent and vice-versa.  Thus, we have a four-dimensional set of parameters $(\alpha,\beta,p,c)\in {\mathbb R}^4$ and there are contours $K(\alpha,\beta,p,c)$ defined by (\ref{eq:4.2}). If a signal $f(x)$ is transient (vanishes outside a finite interval), then functions $\Psi_{\alpha\,\!\beta\,\!\pi}(x,\sigma)$ are defined for every $p$, $\alpha$ and $\beta$. Otherwise, we assume $p\ge 0$ for $\beta=0$ and $p>0$ for $\beta\neq 0$.  

\begin{definition}\label{def4.1}
Let $M$ be a set of parameters, $M\subset P$. Then a point $m\in M$ is a bifurcation point in $M$ if and only if for every open set $V$, $m\in V\subset M$ , there is a point $m'\in V$ such that contour $K(m)$ defined by (\ref{eq:4.2}) for $m$ is not topologically equivalent to $K(m')$.  
\end{definition}
Notice that if $M\subset N$, then a point $m\in M$ may be a bifurcation point in $N$ but not in $M$. Bifurcation points divide $M$ into regions such that every two parameters of the same region define two topologically equivalent contours (and therefore equivalent trees). Thus, instead of trying the impossible - to construct a tree for every parameter value, we can construct a tree for each region and obtain the maximal set of invariants for the given signal. In this section, we will take a somewhat simplified approach. Let us define a one-parametric family of kernels : 
\begin{equation}
  \alpha\Lambda_p^e(x\rho)+\beta\Lambda_p^o(x\rho)\, ,
\end{equation}
where $\alpha$ and $\beta$ are fixed, and $p$ is a varying parameter. Also, let us fix $c\neq 0$.  

First, we will identify all bifurcation points on the $p$-axis. This will give us the maximal set of trees for the given signal and $\alpha$ and  $\beta$. In particular, when $\alpha=0$  $(\beta=0)$ this will give us all odd (even) tree invariants. Then, we will fix $p$, $\alpha$, and $\beta$ and vary $c$. This will give us bifurcation points on the $c$-axis. Some results will also be obtained using the Morse theory approach (see [13]). Computer calculations based on a Fourier transform approach allow us to find all the bifurcation points and all the non-equivalent trees within a certain range of $p$. It will be convenient to establish a convention that all signals $f(x)$ considered in this section (unless otherwise stated) satisfy conditions of Theorem~\ref{thm2.3}, where $p$ is the maximal value of parameter in the range. All the results are also true for any piece-wise $C^\infty$ transient signal, so our assumption is not too restrictive. 

Consider now an equation $F(x,v)=0$ which defines a manifold $M$ of dimension $n$ in $\mathbb{R}^{n+1}$ (assuming $\nu$ is a one-dimensional parameter and $x$ is an $n$-dimensional vector). Consider a projection $\pi:M\to \mathbb{R}$ of $M$ onto the $\nu$-coordinate. It turns out that under certain assumptions, bifurcation points (i.e. points $\nu_0$ such that for fixed $v$, contour $C_\nu$, defined by $F(x,\nu)=0$ changes topologically when $\nu$ passes through $\nu_0$) coincide with $v$ - coordinates of critical points of function $\pi$. By the critical point of a function on a manifold, we mean a point at which the gradient of the function on this manifold is zero. If $M$ is compact, the result quoted above is very well-known. Unfortunately, it is not true for a non- compact manifold $M$. In our case, when $\nu$ is a parameter $p$ or $c$ and may get arbitrarily large, the manifold defined by equations of type (\ref{eq:4.2}) is not compact. We, however, can establish the results quoted above since functions $\Psi_{\alpha\,\!\beta\,\!\pi}(x,\sigma)$ possess many nice properties established in Sections~\ref{sec2} and \ref{sec3}.  

We will use the following notation: 
$$
M^a = \{\, p\in M\ :\  f(p)=a\,\}
$$
where $f:M\to\mathbb{R}$ is a real-valued function on a manifold $M$. In the following result (see [13], Theorem~3.1) will be useful for our purposes: 
\begin{align}
&{\rm Let\ }f:M\to\mathbb{R}\ {\rm be\ a\ real,\ smooth\ function\ on\ a\ manifold\ }M. \label{eq:4.3} \\ 
&{\rm Let\ }a<b,\ {\rm and\ suppose\ that\ the\ set\ }f^{-1}[a,b] {\rm consisting\ of\ all}  \notag \\
&p\in M\ {\rm with\ } a\leq  p \leq b\ {\rm is\ compact\ and\ contains\ no\ critical }  \notag\\
&{\rm points\ of\ }f.\ {\rm Then\ }M_a\ {\rm is\ diffeomorphic\ to\ } M_b.  \notag
\end{align}
 (Our notations are slightly different from those of Milnor, but this result is contained in [13]).  

Assume from now on that $\alpha$ and $\beta$ are fixed and denote $\Psi_{\alpha\,\!\beta\,\!\pi}(x,\sigma)$ by $\Psi(x,\sigma,p)$. Then, equation (\ref{eq:4.4}) below defines a two-dimensional manifold $M_c$.
\begin{equation}\label{eq:4.4}
  {\partial^k\,\!\Psi(x,\sigma,p) \over \partial x^k}=c \quad k=0,1,2,\dots
\end{equation}
in a three-parameter space $(x,\sigma,p)\in \mathbb{R}^3$ if we make a generic assumption (adopted  henceforth) that there is no point satisfying (\ref{eq:4.4}), and such that all three partial derivatives of ${\partial^k\,\!\Psi(x,\sigma,p) \over \partial x^k}$ by  $x$, $\sigma$ and $p$ are simultaneously zero at this point.  
\begin{theorem}\label{thm4.1}
(a)  Let $c \neq 0$ and $M_c$ be a manifold in $\mathbb{R}^3$  defined by (\ref{eq:4.4}). Suppose that $c$ is  fixed and that $p$ varies. If $p_0$ is a bifurcation point, then at some point $(x_0,a_0,p_0)\in M_c$ we have:   
\begin{equation}\label{eq:4.5}
  {\partial^{k+1}\,\!\Psi(x_0,\sigma_0,p_0) \over \partial x^{k+1}}={\partial^{k+1}\,\!\Psi(x_0,\sigma_0,p_0) \over \partial x^{k}\,\!\partial \sigma}=0.
\end{equation}
(b) Let $p$ be fixed and let $c$ vary on an interval not containing zero. Then, if $c_0$ is a  bifurcation point in a $c$-ddomain, there exists a point $ x_0,\sigma_0,p$ such that (\ref{eq:4.5}) holds for $p=p_0$  
\end{theorem}
\begin{proof}
(a) Suppose there is no point satisfying (\ref{eq:4.5}) on $M_c$. Since $c\neq 0$ by  Theorem~\ref{thm2.11}, the set $K$ of all points satisfying (\ref{eq:4.4}) where $p$ varies in $[0,p_0+\Delta]$ for some $\Delta>0$ is compact. Therefore, $\exists\,\Delta_1>0$ such that there is no point $x_0,\sigma_0,p$ in $M_c$ satisfying (\ref{eq:4.5}) and such that $p\in [p_0-\Delta_1,p_0+\Delta_1]$, $\Delta_1<\Delta$.
 
Since $K$ is compact, the set $L=\pi^{-1}[\,p_0-\Delta_1,p_0+\Delta_1\,]\subset K$ is also compact,  where $\pi:M_c\to \mathbb{R}$ is the projection on the $p$-coordinate. It is easy to see that if a point  on $M_c$ is critical, it has to satisfy (\ref{eq:4.4}). Thus our assumption shows that there are no critical points in $L$. Therefore, by (\ref{eq:4.3}) contours  (\ref{eq:4.4}) are all diffeomorphic as $p$ varies in $[p_0-\Delta_1,p_0+\Delta_1]$. This is a contradiction. Thus, there exists a point $(x_0,\sigma_0,p_0)$  satisfying (\ref{eq:4.4}). Case (b) is proved entirely similarly.
\end{proof}
Now that we have identified bifurcation points as critical points of a certain function on a manifold, we can apply results of the well-developed Morse Theory [13] which studies critical points of functions on manifolds. By definition, critical point of function $f$ is non-degenerate if and only if the matrix ${\partial^2\,\!f \over \partial x^1\,\!\partial x^2}$ is non-singular ($x^1,x^2$ are some local coordinate system around the critical point). Generically, critical points may be assumed non-degenerate and isolated (see [13]). Obviously, there are only a finite number of isolated points on a compact set. This observation together with Theorem~\ref{thm2.11} immediately implies the following result.
\begin{theorem}\label{thm4.2}
(a) Let $c\neq 0$ be fixed; then generically on the manifold $M_c$ defined by (\ref{eq:4.4}) there is  only a finite number of bifurcations points $p_j$ in each finite range $[0,p_0]$.\newline  
\smallskip

(b) Let $p$ be fixed and $c$ vary. Then the set of all bifurcation points with respect to $c$ is either finite or forms a sequence $c_j$ converging to zero.  
\end{theorem}
Let us fix $c\neq 0$. Let $M_{c\,\!p_0}$ be defined as the set  
\begin{equation}\label{eq:4.6}
M_{c\,\!p_0}=\{\,x,\sigma,p\ :\ p \in [0,p_0],\ {\partial^k \Psi(x,\sigma,p)\over \partial x^k}=c,\,\}
\end{equation}
Then $M_{c\,\!p_0} = \pi^{-1}[0,p_0]$ and is compact. It is therefore a compact differentiable manifold  with a boundary (see [7]). Its boundary $\partial M_{c\,\!p_0}=\pi^{-1}(0)\cup \pi^{-1}(p_0)$ consists of a finite  number of closed curves, topologically equivalent to a circle $S^1$. Let $\pi^{-1}(0)=\cup \{\,A_j : j = 1,\dots,l\,\}$ and $\pi^{-1}(p_0)=\cup \{\,B_j : j = 1,\dots,m\,\}$ where $A_j$ and $B_j$ are closed curves. Let us select $B_j$ and define a coordinate system on $\mathbb{R}^3$ such that $(0,0,p_0)$ is inside $B_j$. Let us define a function $\phi:[0,\Delta)\to \mathbb{R}$, where $\Delta>0$ such that $\phi$ is $C^{\infty}$, $\phi(0)=1$, $\phi^{\prime}(x)>0$, $\phi(x)\to\infty$ as $x\to\Delta$. Then, define function $g(x,\sigma,p)$ for arbitrary $x$, $\sigma$ and $p<p_0+\Delta$ as follows:  
$$
g(x,\sigma,p)={\partial^k \Psi(x,\sigma,p)\over \partial x^k}\,\}\quad\ {\rm for\ }p\leq p_0
$$
and  
$$
g(x,\sigma,p)={\partial^k \Psi(\phi(p)\,\!x,\phi(p)\,\!\sigma,p_0)\over \partial x^k}\,\}\quad\ {\rm for\ }p_0\leq p\leq p_0+\Delta\, .
$$
\newline
Consider now the set $R_j$ defined as  
\begin{equation}\label{eq:4.7}
  R_j=\{\,x,\sigma,p\ :\ g(x,\sigma,p)=c,\ \ p_0\leq p\leq p_0+\Delta \,\}\, .
\end{equation}
Consider a component of $R_j$ containing $B_j$ and add a point with coordinates $(0,0,\Delta)$ (in our new coordinate system) to it. Denote this set by $KJ$. Then, $M_{c\,\!p_0}\cup K$ is a  manifold with a boundary having an extra critical (local maximum) point with respect to $\pi$ at $(0,0,\Delta)$. We can now define sets $K_j$ for each contour $B_j$ such that the function $\pi$ it has a local maximum on $K_j$, and similarly define sets $L_j$ for contours $A_j$ with one local minimum of it contained in each $L_j$. We need to select numbers $\Delta=\Delta(j)$ so that all sets $L_j$ and $K_j$ do not intersect. In this case, it is easy to see that the set  
\begin{equation}\label{eq:4.8}
  Q_{c\,\!p_0}=M_{c\,\!p_0}\cup\,\cup_{j=1}^mK_j\,\cup,\cup_{j=1}^lL_j
\end{equation}
is a compact manifold without a boundary. It is a well known fact [14] that every connected compact two-dimensional manifold in $\mathbb{R}^3$ is homeomorphic to a sphere with $n$ handles attached. A rigorous mathematical definition is given in [14], but intuitively attaching a handle means making two holes in a manifold and joining them with a tube. Sphere $\mathbb{S}^2$ is a sphere with zero handles, and a sphere with one handle is a torus. Obviously, the number of handles attached is a topological invariant. Thus, to every connected compact two-dimensional manifold $M$ in $\mathbb{R}^3$, there corresponds a unique integer $h(M)$ which completely defines the topology of this manifold and is equal to the number of handles attached. Thus, for every $p_0\geq 0 $  which is not a  bifurcation point, there is a finite sequence of integers $\mu_1(p_0),\dots,\mu_k(p_0)$, $k=k(p_0)$ such that $\mu_j(p_0)$ is equal to the number of handles attached to the $j$th component of the manifold $Q_{c\,\!p_0}$ defined by (\ref{eq:4.8}). It is easy to see that the topological type of $Q_{c\,\!p_0}$ can be defined without constructing functions $g(x,\sigma,p)$ as follows: on the set $M_{c\,\!p_0}$ define equivalence classes consisting of all elements of $B_j$ or $A_j$ or single points of $M_{c\,\!p_0}\setminus \partial M_{c\,\!p_0}$. Such a set of equivalence classes with factor topology will give us a compact manifold denoted by $N(p_0)$. The numbers of handles for components of $N(p_0)$ will be denoted by $\mu_1(p_0),\dots,\mu_k(p_0)$ where $k=k(p_0)$ is the number of components of $N(p_0)$.  

Integers $\mu_j(po)$ may experience a jump at a bifurcation point only. From Theorem~\ref{thm4.2} (a) it follows that the set of all bifurcation points $p \ge 0$ forms a sequence $p_j<p_{j+1}<P_{j+2} \cdots$. Therefore, we can form a double-indexed sequence $\mu_{j\,\!k}$ of integers where $\mu_{j\,\!k}$  is the number of handles of the $k$th component of the manifold $N(p)$ for any $p_j<p<p_{j+1}$. This definition is independent of $p$ as soon as $p$ is between the two consecutive bifurcation points $p_j$ and $p_{j+1}$. Indeed from Theorem~3.1 in [13], it follows that manifolds $N(p)$ are of the same homotopy type (see (14] for definition), if $p$ ranges between $p_j$ and $p_{j+1}$. Therefore, the numbers $\mu_{j\,\!k}(p)$ are the same for all $p\in(p_j,p_{j+1})$. If we vary $c$ we obtaine a triple-indexed set of integers $\mu_{j\,\!k\,\!l}$ where  $\mu_{j\,\!k\,\!l}$ is the same as $\mu_{j\,\!k}$ defined before, and where $c$ is taken between the consecutive bifurcation points $c_k$ and $c_{k+1}$. Clearly, integers $\mu_{j\,\!k\,\!l}$ are shape invariants and are the  same for functions $f(x)$ and of $a\,\!f(b\,\!x+c)$.  Thus, we get yet another way to characterize a function via discrete invariants. The tree-classification was considered earlier.  

Let us now prove that critical points of projection $\pi$ of a manifold $M_c$ defined by (\ref{eq:4.4}, on the $p$-coordinate coincide with bifurcation points. 

Let $f:M \to R$ be a smooth function and let $m$ be its critical point. Then $f$ is called non-degenerate if in some coordinate system, the matrix ${\partial^2\,\!f \over partial x_1\,\partial x_2}$ is non-degenerate at $m$. Non-degeneracy is a generic property (see [13]); thus, from the results in [13], it follows that the level sets of function $f$ on a compact manifold $M$ undergo topological change when passing through a non-degenerate critical point. Assuming that all critical points of function $\pi$ on the manifold $M_c$ are non-degenerate, we can combine this result with Theorem~\ref{thm4.1} to get the following : 
\begin{theorem}\label{thm4.3}
With assumptions and notations of Theorem~\ref{thm4.1}, the following are true.\newline
\medskip

(a) $p_0$ is a bifurcation point of manifold $M_c$ if and only if $\pi:M_c\to R$ has a critical  point $m\in M_c$ such that $\pi(m)=p_0$.\newline
\medskip

b) $c_0$ is a bifurcation point of the hyperplane $p=p_0$ if and only if there is a critical point of function ${\partial^k\,\!\Psi \over \partial x^k}(x,\sigma,p)$ which takes on the value $c_0$ on the hyperplane  $p=p_0$. 
\end{theorem}
\begin{proof}
The "only if" part was proved in Theorem~\ref{thm4.1}. Let $p_0$ be a value of $\pi$ at some non-degenerate critical point $m$. Let $p_l>p_0$ be a number such that there are no critical points in $M_c$ taking on values in $(p_0,p_1)$.  If we define a manifold $Q_{c\,\!p_1}$ as in (\ref{eq:4.8})   (changing $p_0$ into $p_1$), then   
\begin{equation}
  \pi^{-1}[0,p_1]\cap Q_{c\,\!p_1}=M_c\cap \pi^{-1}[0,p_1]\, .
\end{equation}
Since $Q_{c\,\!p_1}$ is compact, we can apply the above mentioned result in [13] and conclude that $m$ is a bifurcation point of $\pi$ on $Q_{c\,\!p_1}$ and, therefore, of $\pi$ on $M_c$. This proves (a). The proof for (b) is entirely similar.  
\end{proof}
Thus, bifurcation points and critical points coincide. However, there is a theory relating the number of critical points on a manifold and a topology of this manifold. For compact manifolds, good references are [7] and [13]. For non-compact manifolds, the theory is much less developed and, in general, the results for compact manifolds cannot be generalized for non-compact ones. Some interesting results, however, were obtained in [9] and [15]. Before we proceed to the Morse theory applications, let us introduce some definitions. Let $m$ be a critical point of a function $f:M\to R$; then, the number of negative eigenvalues of the Hessian $\big({\partial^2\,\!f \over \partial x_1\,\partial x_2}\big)$ at $m$ will be called the index of $f$ at $m$. Obviously, the index of $f$  at $m$ is zero for a local minimum, and is equal to the dimension of $M$ at a local maximum. Let $X$, $A$ be a pair of topological spaces $A\subset X$. Denote by $H_j(X,A)$ the $j$th homology group of the pair with coefficients in the group of integers $\mathbb{Z}$ (see [14] for definitions). The rank of $H_j(X,A)$ is called Betti number $\beta_j(X,A)$. If $A$ is empty, we simply write $H_j(X)$ instead of $H_j(X,\emptyset)$. The basic results of Morse theory state that if $f$ is a function $f:M \to R$ on a compact manifold $M$ having no degenerate critical points, then: 
\begin{align}
  c_j &\ge \beta_j(M)\, ,\label{eq:4.9} \\
\sum_{j=1}^n(-1)^j c_j &= \sum_{j=1}^n(-1)^j\beta_j(M)\, ,\quad n=\dim M \, ,\label{eq:4.10}
\end{align}
where $c_j$ is the number of critical points of function $f$ having index $j$ (see [13] for the proof).  
\begin{theorem}
Let $p_0$ be a point on a $p$-axis which is not a bifurcation point. Let $M=M_{c\,\!p_0}=\pi^{-1}[0,p_0]$ be manifold with a boundary $\partial\,\! M =\pi^{-1}(0)\cup \pi^{-1}(p_0)$. ($M_{c\,\!p_0}$  is also defined by (\ref{eq:4.6})). Assume all critical points of $\pi$ on $M_{c\,\!p_0}$ are non-degenerate. Let $k$ be the number of closed contours in $\pi^{-1}(p_0)$ and let $l$ be the number of closed contours in $\pi^{-1}(0)$. Denote by $c_j$ $j=0,1,2$ the number of bifurcation points of index $j$ of the function $\pi$ on $M=M_{c\,\!p_0}$. Then, if the compact manifold $Q_{c\,\!p_0}$ is defined by (\ref{eq:4.8}), we have the following : 
\begin{align}
 & c_2\ge \beta_2(Q_{c\,\!p_0})-k, \qquad c_1\ge \beta_1(Q_{c\,\!p_0}), \qquad c_0\ge \beta_0(Q_{c\,\!p_0})-l \label{eq:4.11} \\
 & c_2-c_1+c_0 = \beta_2(Q_{c\,\!p_0})-\beta_1(Q_{c\,\!p_0})+\beta_0(Q_{c\,\!p_0})-l-k\, . \label{eq:4.12}
\end{align}
Also, let $d_0$ be the number of connected components of $M$ that do not intersect $\pi^{-1}(0)$ , andlet $d_2$ be the number of connected components of $M$ that do not intersect $\pi^{-1}(p_0)$.  Let $r_1$  be the number of components of $M$ and let $r_2$ be the number of components of $M$ that do not intersect $\partial\,\!M$. Then 
\begin{equation} \label{eq:4.13}
c_2 \ge d_2,\quad c_1 \ge \beta_1(M,\partial\,\!M)-k-l+r_1-r_2,\quad c_0 \ge d_0\, .    
\end{equation}
\end{theorem}
\begin{proof}
According to (\ref{eq:4.9}) applied to $M=Q_{c\,\!p_0}$, the total number of critical points (= bifurcation points) of index $j$ is $c_j \ge \beta_j(Q_{c\,\!p_0})$. However, there are $k$ maxima of $Q_{c\,\!p_0}$ and $l$-minima which are not bifurcation points of $M$. Since minima have index $0$ and maxima index $2$, (\ref{eq:4.11}) follows. Equality (\ref{eq:4.12}) follows directly from (\ref{eq:4.10}). Let us prove (\ref{eq:4.13}). If a component of $M$ does not intersect $\pi^{-1}(0)$, then its local minimum (which is a  critical point) does not lie in $\pi^{-1}(0)$. Thus, it is a bifurcation point of index $0$ on $(0,p_0)$. This proves $c_0 \ge d_0$. Inequality $c_2 \ge d_2$ is proved similarly.  

Let us prove the middle inequality in (\ref{eq:4.13}). For this we need the following result (see [16]):  

Let $A\subset X$ and $X,A$ be a pair of compact spaces, $A$ being an $ANR$ (see (16] for  definition), in particular, $A$ is a manifold, and let $\mu:X\to \tilde{X}$ be a factor map of $X$ onto a  space $\tilde{X}$ obtained as a factor space from $X$ by collapsing every connected component of $A$ into a point. Let $\tilde{A} = \mu(A)$. Then $\mu$ induces an isomorphism  $\mu_*\,:\,H_j(X,A) \to H_j(\tilde{X},\tilde{A})$ of homology groups.

If we perform the operation of collapsing components with the pair $M, \partial\,\!M$, we will get $\mu(M)=Q_{c\,\!p_0}$ (this was mentioned earlier).  Then $\mu(\partial\,\!M)$ is a union of $k+l$ points. Thus, $H_1(M,\partial\,\!M)=H_1(Q_{c\,\!p_0},F)$ where the set $F$ consists of $k+l$ points.   From the exact sequence for homology of the pair $Q_{c\,\!p_0},F$,  
$$
H_1(F)=0\to H_1(Q_{c\,\!p_0}) \to H_1(Q_{c\,\!p_0},F) \to  H_0(F) \to  H_0(Q_{c\,\!p_0}) \to H_0(Q_{c\,\!p_0},F) \to 0
$$   
we get: 
$$
\beta_1(Q_{c\,\!p_0},F)=\beta_1(Q_{c\,\!p_0})+\beta_0(F)-\beta_0(Q_{c\,\!p_0})+\beta_0(Q_{c\,\!p_0},F)\, .
$$  
However, it is easy to see that 
$$
\beta_0(F)=k+l, \qquad \beta_0(Q_{c\,\!p_0})=r_1, \qquad \beta_0(Q_{c\,\!p_0},F)=r_2\, .
$$ 
Thus 
$$
 \beta_1(M,\partial\,\!M)=\beta_1(Q_{c\,\!p_0},F)=\beta_1(Q_{c\,\!p_0})+k+l-r_1+r_2;
$$
from (\ref{eq:4.13})
$$
 c_1\ge \beta_1(Q_{c\,\!p_0})=\beta_1(M,\partial\,\!M)-k-l+r_1-r_2\, .
$$
\end{proof}

Numbers $d_2$, $d_0$, $k$, $l$, $r_1$, $r_2$, and $\beta_1(M,\partial\,\!M)$ can be calculated easily by considering contours ${\partial^k\,\!\Psi(x,\sigma,p) \over partial x^k}=c$ for fixed $p$. We can show it in the following example. Note that as $p$ increases, critical points of index 2 (maxima) correspond to  shrinking contours which reduce to a single point at a bifurcation point; critical points of index  0 (minima) correspond to contours appearing out of a single point, and contours corresponding  to saddle bifurcation points (index 1) correspond to contours that join or split into two.
 Figures~3-6  show convolutions of a function with kernels $\rho^{p+1}\nabla_p^e(x\rho)$ for $p=2.5,\,3.0,5,\,5.5$.  Clearly, there are bifurcation points of index $1$ in $[2.5,3]$ and in $[5,5.5]$. 

\begin{center}
\includegraphics[width=15cm]{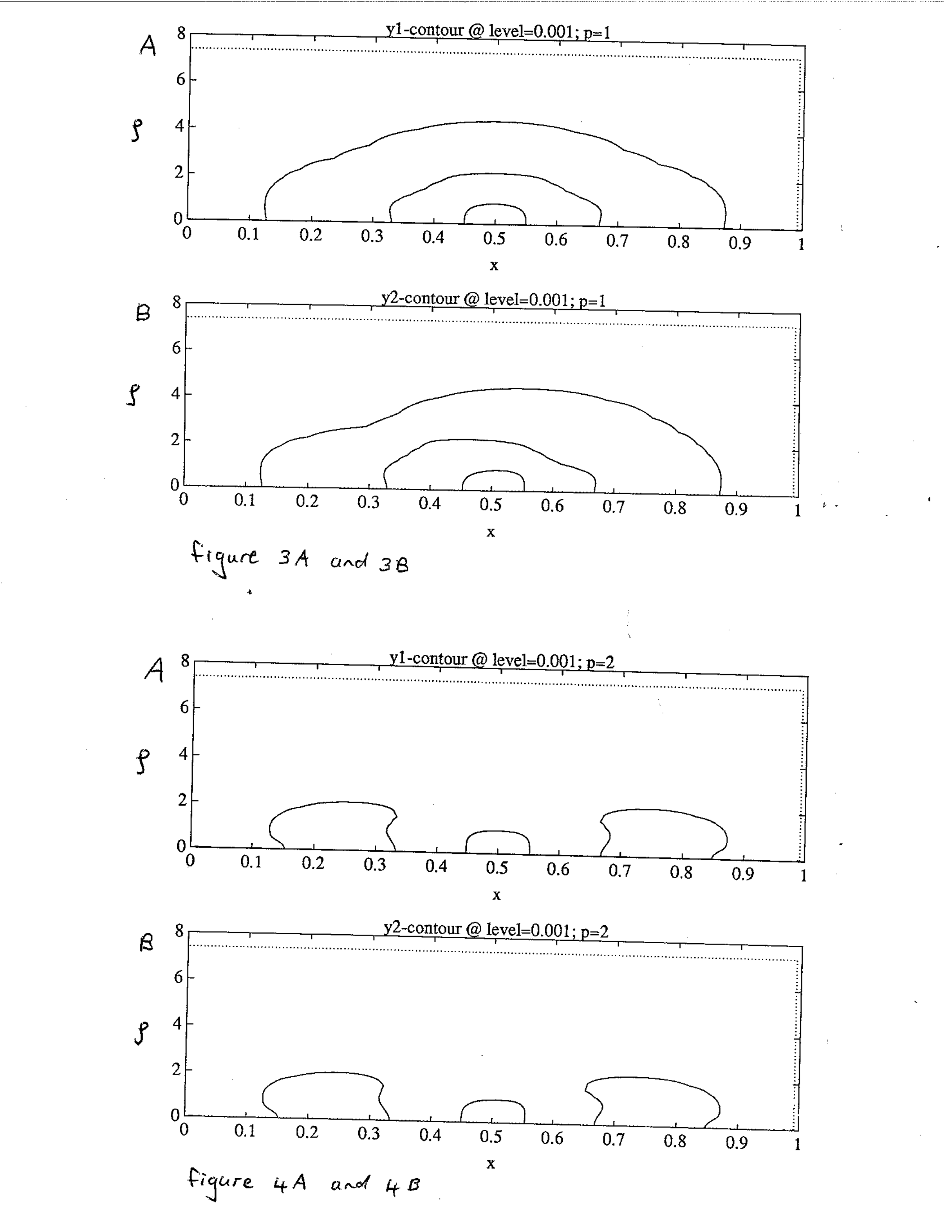}
\includegraphics[width=15cm]{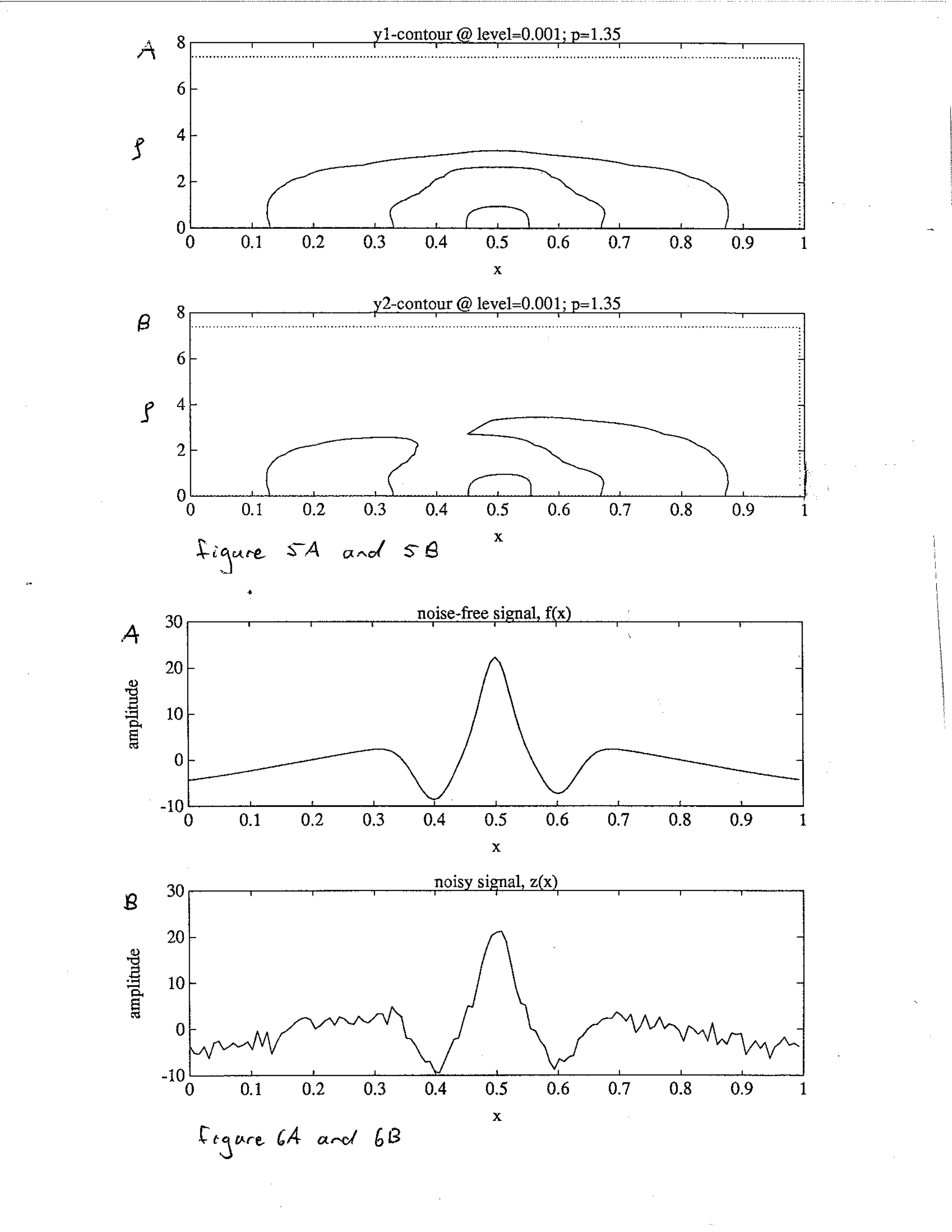}
\end{center}

 \section{Conclusion}
\label{sec:conclusion}
  \label{sec5}
We have studied a three-parametric family of kernels $\alpha \rho^{p+1}\nabla_p^e(x\rho)+\beta \rho^{p+1}\nabla_p^o(x\rho)$ which describes all possible monotonic kernels (see [3]). This gave us a way to construct two different types of tree invariants for a given signal. Various Fourier transform techniques were developed for quick computations.  

We also gave a rigorous treatment of the structure of level crossings of functions ${\partial^k\,\!\Psi(x,\sigma,p) \over partial x^k}$ which, in particular, generalizes many assumptions about Gaussian  convolutions (most of them have never been proved before). Thus, we have given a solid foundation to many papers for a particular case of Gaussian convolution $(\nabla_p^e*f)$ [1, 2, 11, 17, 18]. Another important application is assigning discrete values to characterize the signal developed in Section~\ref{sec4}, where we have introduced integers  $\mu_{j\,\!k}$ and $\mu_{j\,\!k\,\!l}$ to describe a signal. This is an alternative, geometric approach to signal classification. There are two other conceptual developments presented in Sections~\ref{sec2} and \ref{sec3} that we find very promising. First, the introduction of yet one more definition of fractional derivatives and, connected with it, a geometric method of analyzing the signal. Secondly, we can extend the developed theory to more general kernels along the lines of Theorem~\ref{thm3.4}. Using a differential geometric approach to signal processing seems to give interesting results. We have developed this approach in Section~\ref{sec4} where Morse theory is successfully applied to signal characterization. Many applications of the developed theory include signal characterization and analysis edge detection and underwater signal recognition.  
\section{References}
\label{sec:refs}

[1] Babaud, J., Witkin, A.P., Baudin, M., and Duda, R., "Uniqueness of the Gaussian Kernel for Space Filtering," IEEE Trans. on PAMI, PAMI-8(1), pp.26-33 (Jan. 1986), doi:10.1109/TPAMI.1986.4767749\newline
[2] Witkin, A.P., "Scale Space Filtering: A New Approach to Multiscale Description," in Image Understanding, ed. S. Ullman and W. Richards, Norwood, NJ: Ablex (1984), doi:10.1109/ICASSP.1984.1172729\newline
[3] Luxemburg, L.A. and Damelin, S. B., "A Multiple Parameter Linear Scale-Space for one dimensional Signal Classification." preprint \newline
[4] Shilov, G.E., "Mathematical Analysis, Functions of One Variable - Part III",  Moscow: Nauka (1970). (Russian).\newline
[5] Shilov, G.E., "Elementary Real and Complex Analysis, Cambridge": MIT Press (1973).\newline
[6] Oldham, K. and Spanier, T., "The Fractional Calculus", New York: Academic Press (1974).\newline
[7] Hirsch, M., "Differential Topology", New York: Springer-Verlag (1976).\newline
[8] Palis, T., Jr. and M, Wellington, "Geometric Theory of Dynamical Systems", New York: Springer-Verlag (1982).\newline
[9] Luxemburg, L.A., "Structural Stability Analysis and its Applications to Power Systems,"  Ph.D. Dissertation, Texas A\&M University, College Station, TX (Dec. 1987).\newline
[10] Coddington, E.A. and Levinson, N., "Theory of Ordinary Differential Equations", New York: McGraw-Hill (1955).\newline
[11] Yuille, A.L. and Poggio, T., "Scaling Theorems for Zero Crossings," IEEE Trans. Pattern  Anal. Machine Intell., 8(1), pp.15-25 (1986),\newline doi:10.1109/TPAMI.1986.4767748\newline
[12] Clark, T.T., "Singularity Theory and Phantom Edges in Scale Space," IEEE Trans. on  PAMI, 10(5), pp.720-727 (Sep. 1988).\newline
[13] Milnor, T., "Morse Theory," Annals of Mathematics, V51, Princeton University Press, Princeton, NJ (1963).\newline
[14] Spanier, E., "Algebraic Topology", New York: McGraw-Hill (1966).\newline
[15] Luxemburg, L.A. and Huang, G.M., "Generalized Morse Theory and its Applications to Control and Stability Analysis," CSSP Journal,\newline doi:10.1007/bf01183770\newline
[16] Dold, A., "Lectures on Algebraic Topology", New York: Springer-Verlag (1980).\newline
[17] Mokhtarian, F. and Mackworth, A., "Scale-Based Description and Recognition of Planar Curves and Two-Dimensional Shapes," IEEE Trans. on PAMI, PAMI-8(l), pp.34-43 (Jan. 1986).\newline
[18] Clark, T.T., "Singularities of Contrast Functions in Scale Space," Proc. 1st Int. Conf. Computer Vision, London, pp.491-496 (1987).\newline
[19] Thareja, S., Rohde, G.,  Martin, R.D.,  Medri, I. and Aldroubi, A;  "Signed Cumulative Distribution Transform for Parameter Estimation of 1-D Signals", arXiv:2207.07989 . \newline
[20] Aldroubi, A., Huang, L., Kornelson, K. and Krishtal, I .,"The Signed Cumulative Distribution Transform for 1-D Signal Analysis and Classification", 
arXiv:2106.02146. \newline
[21] Aldroubi, A., Gröchenig, K., Huang, L., Jaming, P.,  Krishtal, I. and  Romero, R.L., "Sampling the flow of a bandlimited function", arXiv:2004.14032.\newline
[22] Damelin, S.B. and Miller, W., "Mathematics of Signal Processing", Cambridge Texts in Applied Mathematics (No. 48) February 2012. \newline
[23] Damelin, S.B, Guo, H. and Miller, W., "Solutions to S. B. Damelin and W. Miller, Mathematics and Signal Processing", in Resources: Mathematics and Signal Processing, Cambridge Texts in Applied Mathematics (No. 48) February 2017. \newline
[24] Benedetto, J.J. and Dellomo, M, R.,"Reactive Sensing and Multiplicative Frame Super-resolution", arXiv:1903.05677. \newline
[25] Andrews, T.D., Benedetto, J.J. and Donatelli, J.J., "Frame multiplication theory and a vector-valued DFT and ambiguity function", arXiv:1706.05579.
[26] Candes, E.J. and Plan, Y., "A probabilistic and RIPless theory of compressed sensing", arXiv:1011.3854.
\end{document}